\documentclass[11pt]{amsart}
\usepackage{amsmath,amssymb,amsthm,mathrsfs,mathtools}
\usepackage{booktabs,enumitem}
\usepackage{tikz}
\usetikzlibrary{arrows.meta,backgrounds,positioning}
\usepackage[margin=1.15in]{geometry}
\usepackage[hypertexnames=false,colorlinks=true,breaklinks=true,bookmarks=false,urlcolor=blue,citecolor=blue,linkcolor=blue,bookmarksopen=false,draft=false]{hyperref}

\makeatletter\renewcommand*{\top}{%
{\mathpalette\@transpose{}}%
}
\newcommand*{\@transpose}[2]{%
\scriptsize
\raisebox{\depth}{$\m@th#1\mathsf{T}$}%
}
\makeatother

\newtheorem{theorem}{Theorem}[section]
\newtheorem{lemma}[theorem]{Lemma}
\newtheorem{proposition}[theorem]{Proposition}
\newtheorem{corollary}[theorem]{Corollary}
\theoremstyle{definition}
\newtheorem{definition}[theorem]{Definition}

\theoremstyle{remark}
\newtheorem{remark}[theorem]{Remark}
\newtheorem*{remark*}{Remark}

\newcommand{\PP}{\mathscr{P}}
\newcommand{\QQ}{\mathscr{Q}}
\newcommand{\vol}{\operatorname{vol}}
\newcommand{\OC}{\mathrm{OC}}
\newcommand{\clo}{c_{\mathrm{lo}}}
\newcommand{\chii}{c_{\mathrm{hi}}}
\newcommand{\den}{\operatorname{den}}
\newcommand{\R}{\mathbb{R}}

\begin{document}

\title[Volumes of Boolean quadric polytopes of series-parallel graphs]{Exact volume computation for Boolean quadric polytopes of series-parallel graphs}

\author{Jon Lee}
\address{University of Michigan, Ann Arbor, MI 48109, USA}
\email{jonxlee@umich.edu}
\urladdr{\href{https://orcid.org/0000-0002-8190-1091}{https://orcid.org/0000-0002-8190-1091}}


\begin{abstract}
For a graph $G$, the Boolean quadric polytope $\PP(G)$ is the convex hull of the binary
solutions of $y_{ij}=x_ix_j$ for $ij\in E(G)$, and $\QQ(G)$ is its standard linear relaxation. For series-parallel graphs, $\QQ(G)$ together with the so-called odd-cycle
inequalities describes $\PP(G)$. 
Lee and Skipper described how $\vol\QQ(G)$ can be computed in polynomial time for graphs of bounded treewidth and gave a closed formula for $\vol\PP(G)$ when $G$ is a cycle. 
We resolve their open question of giving an efficient algorithm to compute $\vol\PP(G)$ for series-parallel
graphs. Letting $d=|V(G)|+|E(G)|$, for a series-parallel graph $G$ we can compute $\vol\PP(G)$ with $O(d^7)$ arithmetic
operations, and with $O(d^5)$ operations if $G$ is a cactus, on rationals of bit length
$O(d^2)$.

The algorithm is a dynamic program over the series-parallel decomposition.
The same framework computes $\vol\QQ(G)$ (and hence the number of linear extensions of
the incidence poset of $G$), with $O(d^4)$ arithmetic operations, and with $O(d^3)$ operations if
$G$ is a cactus, on rationals of the same bit length. With every vertex coordinate $x_v$ fixed at
$\tfrac12$, where the recursion reduces to convolutions of univariate polynomials, it also
computes the volume of the cut polytope of every series-parallel graph with $m$ edges, with
$O(m^3)$ arithmetic operations.

We also prove results on how much of $\QQ(G)$ the polytope $\PP(G)$ occupies. First, short odd
cycles provably matter more than long ones: in every graph, the odd-cycle inequalities of a cycle
of length $\ell$ cut off at most a fraction $2^{\ell-1}/\ell!$ of $\QQ(G)$, and this bound is sharp. Hence
$\vol\PP(G)/\vol\QQ(G)\ge1-\sum_C2^{|C|-1}/|C|!$ for every series-parallel graph, where $C$ ranges
over the cycles of $G$. The ratio
$\vol\PP/\vol\QQ$ does not factor over cycles that share a vertex: in a flower of $k$ copies of a
cycle $C_\ell$ (the $k$ copies share one vertex and are otherwise disjoint) it decays like $\rho_\ell^k$, where $\rho_\ell$ is an explicit rational number
below the ratio of $C_\ell$ itself, and a triangle glued to a long cycle keeps, in the limit, a
fraction $\pi^2(\pi^4+45)/(\pi^6+480)$ of its own ratio. Nevertheless, the small fractions that many long cycles cut off compound: in the worst case, the
triangle inequalities, or the odd-cycle inequalities of cycles up to any fixed length, close no
fixed fraction of the gap between $\QQ(G)$ and $\PP(G)$, and the volume ratio can be
exponentially small in $d$.
\end{abstract}

\maketitle

% ---------------------------------------------------------------------------
\section{Introduction}

Let $G=(V,E)$ be a simple graph with $n=|V|$ and $m=|E|$. The (graphical) \emph{Boolean
quadric polytope} $\PP(G)\subset\R^{V\cup E}$ is the convex hull of the vectors $(x,y)$ with
$x\in\{0,1\}^V$ and $y_{ij}=x_ix_j$ for $ij\in E$. Its standard relaxation $\QQ(G)$ is the
solution set of
\begin{equation}\label{eq:Q}
y_{ij}\ge0,\qquad y_{ij}\le x_i\,,\qquad y_{ij}\le x_j\,,\qquad x_i+x_j\le1+y_{ij}\qquad(ij\in E).
\end{equation}
These polytopes are central to quadratic $0/1$ optimization and, through a linear isomorphism,
to the cut polytope of the graph obtained from $G$ by adding a vertex adjacent to all others;
see~\cite{Padberg89,DeSimone90,DezaLaurent97}. Following~\cite{LeeMorris94}, we
regard volume as a natural measure of how well a relaxation approximates a polytope.
The volume $\vol\QQ(K_n)$, equivalently that of the rooted metric polytope, was computed
in~\cite{KLS97}. Sparse graphs were studied
in~\cite{LeeSkipper20}, whose authors
\begin{itemize}[nosep]
\item showed that $\vol\QQ(G)$ can be computed in polynomial time when $G$ has bounded
treewidth, via counting linear extensions (see~\cite{Stanley86,KHNK16});
\item gave closed forms for $\vol\QQ$ of stars, paths, and cycles, and for
$\vol\PP(C_m)$;
\item showed experimentally that odd-cycle inequalities of short cycles matter far more
than those of long cycles; we prove a strong form of this in Theorem~\ref{thm:cap}.
\end{itemize}
The main open problem of~\cite[Section~8]{LeeSkipper20} was to find a polynomial-time algorithm
for $\vol\PP(G)$ when $G$ is series-parallel, that is, when $G$ has no $K_4$ minor, with cactus
forests as a more manageable first case. Lee, Skipper, and Speakman restated this problem
as~\cite[Problem~20, ``Lobster Roll'']{LSS18}, noting that, beyond cycles, nothing was known about the complexity
of computing $\vol\PP(G)$ for graphs that are not forests. For series-parallel graphs,
\cite{Padberg89} showed that $\QQ(G)$ together with the odd-cycle inequalities
describes $\PP(G)$; however, there can be exponentially many such inequalities, and the
regions that they cut off overlap. The problem is not easy in general: by a theorem of~\cite{DittmerPak20}, counting linear extensions of incidence posets is \#P-complete, so even
computing $\vol\QQ(G)$ is \#P-hard for arbitrary graphs (see~\cite[Section~4]{LSS18}). We resolve the
problem for series-parallel graphs, and hence, in particular, for cactus forests: our main
result, Theorem~\ref{thm:main}, is an algorithm that computes $\vol\PP(G)$ exactly for every
series-parallel graph $G$, with a number of arithmetic operations polynomial in $d=n+m$, namely
$O(d^7)$ in general and $O(d^5)$ for cacti, on rationals of bit length $O(d^2)$.

\subsection*{Approach}
We think of volumes as masses of distributions. A point of $\QQ(G)$ consists of vertex coordinates
$x\in[0,1]^V$ and edge coordinates $y\in\R^E$, so the uniform distribution on $\QQ(G)$ makes $(x,y)$
a random point, and $\vol\PP(G)/\vol\QQ(G)$ is the probability that this point lies in
$\PP(G)$. Fixing the vertex coordinates $x$ leaves a \emph{fiber}, the set of admissible $y$. Its
volume, as a function of $x$, is the density of the vertex coordinates, up to normalization, and
the volume of the polytope is the integral of this density over $x$. The same holds for $\PP(G)$,
and for any piece of $G$ once the coordinates at its boundary are fixed. The algorithm works with
these unnormalized densities, which are the fiber volumes, rather than with the polytopes
themselves. The picture is especially simple in the cut coordinates $z_e=x_i+x_j-2y_e$\,, which at
the $0/1$ points of $\PP(G)$ equal $|x_i-x_j|$ and so record whether $e$ joins the two sides of
the cut $\{i:x_i=1\}$: once $x$ is fixed, each $z_e$ ranges over its own interval, independently of the others
(Lemma~\ref{lem:edge}), so the fiber of $\QQ(G)$ is a box. Inside the box, $\PP(G)$ is described by
lengths. Label each edge of a path or cycle \emph{even} or \emph{odd}; an even edge $e$ has length
$z_e$ and an odd one has length $1-z_e$, the labelled path or cycle has the sum as its length, and
the labelling is \emph{odd} if it has an odd number of odd edges. A point lies in $\PP(G)$ exactly
when every odd labelling of every cycle has length at least $1$. This is the parity shortest-path
view behind the separation algorithm of~\cite{BarahonaMahjoub86}.

A series-parallel graph is built from single edges by series and parallel compositions of
two-terminal pieces. For a piece $H$ with terminals $s$ and $t$, the rest of the graph sees $H$
only through $x_s$\,, $x_t$\,, and two numbers: the \emph{parity distances} $d_0(H)$ and $d_1(H)$,
the lengths of the shortest $s$--$t$ paths in $H$ with an even and with an odd labelling. So the
algorithm never enumerates cycles. For each
$(x_s,x_t)$, it records the volume of the feasible configurations of $H$, spread out according to
the values of $d_0(H)$ and $d_1(H)$, that is, an unnormalized distribution of these two lengths; we
call this record the \emph{state} of $H$. Two facts make states small. If $s$ and $t$ are joined by
an edge of $H$, then $d_0=z_{st}$ and $d_1=1-z_{st}$\,, so the state is a function of three
variables. Otherwise the mass of $H$ and the density of $d_0$ suffice, because a symmetry gives the
density of $d_1$ from that of $d_0$.

\emph{Series composition.} A series composition of pieces, joined end to end at shared vertices,
is a \emph{chain}, and the pieces are its \emph{segments}. An $s$--$t$ path through a chain passes
through its segments one after another, so its length is the sum of the lengths of its parts in
the segments. If the vertex coordinates inside the chain
were fixed, the segments would be independent, because they share no edge coordinates, and the
distribution of a sum of lengths would be the convolution of the distributions of the segments,
just as the density of a sum of independent random variables is the convolution of their
densities. The internal vertex coordinates are not fixed; each is shared by the two segments that
meet there, so a series step multiplies the states of the segments, convolves in the lengths, and
integrates over the shared vertex coordinates. Closing the chain by the edge $st$ creates cycles,
and the configurations that violate one of them are those on which some odd labelling of a new
cycle has length less than $1$; their volume is the value at $1$ of a cumulative
distribution of a sum, one for each odd labelling of the segments and the edge $st$. The $2^k$
odd labellings of a chain with $k$ segments give disjoint sets, and a symmetry reduces them to two (Theorem~\ref{thm:chain} and
Lemma~\ref{lem:chainswitch}). In rotated coordinates, the coupling through the shared vertex
coordinates becomes a box, and each step is a convolution in the plane
(Proposition~\ref{prop:box}).

\emph{Parallel composition.} Parallel branches share only $s$ and $t$, so, for fixed $(x_s,x_t)$,
their configurations are independent and their masses multiply. The new cycles run out along one
branch and back along another, and each parity distance of the composite is the smallest over
its branches. Feasibility compares the even distance of each branch with the odd distances of the
others,
so the new state is computed as for the minimum or maximum of independent random variables: sum,
over the branch that attains the extreme value, its density times the probability that the other
branches satisfy the constraints. This is a single integral in one length variable
(Theorem~\ref{thm:theta}).

\emph{Assembling the volume.} Blocks that meet at a cut vertex share only its coordinate, so,
for that coordinate fixed, they too are independent: their fiber volumes multiply, and the
product is integrated over the shared coordinate (Section~\ref{sec:cutvertex}). At the root, the
remaining variables are integrated out, which gives the volume. For $\QQ(G)$ there are no lengths,
and each step only multiplies and integrates over vertex coordinates, a continuous analogue of
multiplying transfer matrices; this recursion uses $O(d^4)$ arithmetic operations, and $O(d^3)$
for cacti, compared with $O(d^6)$ for the general bounded-treewidth algorithm that Lee and Skipper
applied, and it counts the linear extensions of the incidence poset of $G$
(Section~\ref{sec:volQ}). If instead every vertex coordinate is fixed at $x_v=\tfrac12$ and not
integrated out, the segments are independent, every step is a convolution of univariate
polynomials, and the recursion computes the volume of the cut polytope of $G$
(Section~\ref{sec:cutG}).

What makes all this efficient is Theorem~\ref{thm:allpoly}: every state is a \emph{single}
polynomial on a fixed domain, so the pieces of a piecewise-polynomial description never
accumulate as the recursion proceeds, and each step is ordinary polynomial arithmetic. For a
piece whose terminals are joined by an edge, the domain is the simplex $\Delta$ of
Lemma~\ref{lem:edge}; the state is the fiber function of a $0/1$ polytope whose vertices map onto
the vertices of $\Delta$, so it can change form only across the facets of $\Delta$. For a
density, the only nonpolynomial ingredient of the integrals is an integration limit, which
disappears when we differentiate. Together with a bound on bit lengths
(Proposition~\ref{prop:bits}), this yields Theorem~\ref{thm:main}.

\subsection*{What the volumes say}
The algorithm makes $\vol\PP(G)/\vol\QQ(G)$ computable, and the coordinates behind it, cut
coordinates and fibers over single vertices, also yield exact statements about this ratio
(Section~\ref{sec:comp}). The main one
confirms, in a strong form, the heuristic of Lee and Skipper that odd-cycle inequalities matter
far more for short cycles than for long ones: \emph{in every graph, the odd-cycle inequalities of a
cycle of length $\ell$ cut off at most a fraction $2^{\ell-1}/\ell!$ of $\QQ(G)$}, and the bound is
sharp, in the limit, for a cycle with
many pendant edges (Theorem~\ref{thm:cap}). Adding these bounds over the cycles gives
$\vol\PP(G)/\vol\QQ(G)\ge1-\sum_C2^{|C|-1}/|C|!$ for every series-parallel graph. A triangle hanging from a vertex of any graph always cuts
off at least $1/3$, as much as any $4$-cycle can, and $5/2$ times as much as any longer cycle can
(Corollary~\ref{cor:pendanttri}). More generally, a cycle of length $\ell$ hanging from a vertex of a
series-parallel graph matters, up to explicit constant factors, exactly as much as it does on its
own, so its contribution falls like $\pi^{2\ell}/(2\ell)!$ (Theorem~\ref{thm:hangcyc}). For a single cycle the ratio is known in closed form, and for
disjoint cycles the ratios multiply. Once cycles
share a vertex, they do not. The fiber marginals of $\PP$ and $\QQ$ at a vertex are unimodal
(Lemma~\ref{lem:unimodal}), and in a flower of $k$ copies of $C_\ell$\,,
the ratio is $\rho_\ell^k$ up to a factor linear in $k$, with $\rho_\ell=1-\binom{2\ell-2}{\ell-1}/A_{2\ell-1}$\,, where
$A_k$ counts alternating permutations; $\rho_\ell$ is strictly less than the ratio of $C_\ell$
alone (Theorem~\ref{thm:bigflower}, Proposition~\ref{prop:rhocenter}). A long cycle glued to a
triangle loses almost nothing itself, yet it multiplies the triangle's ratio by a factor tending
to $\pi^2(\pi^4+45)/(\pi^6+480)\approx0.9751$ (Theorem~\ref{thm:longcycle}). The reason is that,
as the cycle grows, the distribution that it induces on the shared coordinate $x_v$ converges to
the density $2\sin^2(\pi x_v)$, which is concentrated near $x_v=\tfrac12$, exactly where the
odd-cycle inequalities of the triangle cut off most of their volume; we obtain this limit from a
spectral formula for the fibers of $\QQ$ of a cycle over one vertex.
Finally, the triangle inequalities close all of the gap between $\QQ(G)$ and $\PP(G)$ exactly
when every chordless cycle is a triangle (Proposition~\ref{prop:chordless}), and, because the small fractions that many long cycles cut off
compound, in the worst case they close no fixed fraction of it: on flowers of $4$-cycles, $\vol\PP$ is a fraction $(0.989\ldots)^{d+o(d)}$ of the volume of the
triangle relaxation, and the same holds, with other constants, for the odd-cycle inequalities of
cycles up to any fixed length (Theorem~\ref{thm:worst}).

\subsection*{Notation}
For a graph $G$, we write $V(G)$ and $E(G)$ for its vertex and edge sets, and $ij$ for the edge
$\{i,j\}$. We write $K_n$ for the complete graph on $n$ vertices, $C_m$ for the cycle with $m$
edges, $K_{a,b}$ for the complete bipartite graph, and $G-e$ for $G$ with the edge $e$ deleted.
Throughout, $d=|V|+|E|$, and $\vol$ denotes $d$-dimensional Lebesgue volume; we define the
variant $\vol_z$ in Section~\ref{sec:prelim}. We use $\subset$ for (not necessarily proper)
inclusion of sets. We write $|A|$ for the cardinality of a set $A$ and $A\triangle B$ for the
symmetric difference of sets $A$ and $B$. For a cycle $C$ and a vertex $v$ of $C$, $\delta_C(v)$
denotes the set of the two edges of $C$ at $v$. For points $p,p'\in\R^2$, $p\le p'$ means
$p_1\le p'_1$ and $p_2\le p'_2$\,. For a polytope $K$, $\operatorname{vert}K$ denotes its vertex
set, and $K$ is a \emph{$0/1$ polytope} if every vertex is a $0/1$ vector. A \emph{pendant edge} is an edge with an end of degree $1$. A \emph{cactus} is a connected graph in which every edge lies in at most one cycle, and a
\emph{cactus forest} is a graph whose components are cacti. Finally, $\log$
denotes the logarithm to base $2$, and $\lfloor\cdot\rfloor$ denotes the floor function.

\subsection*{Organization}
In Section~\ref{sec:prelim}, we introduce cut coordinates, in which $\QQ(G)$ becomes a product of
intervals once the vertex coordinates are fixed and each odd-cycle inequality becomes a lower
bound on the length of an odd labelling of a cycle; we also recall series-parallel decompositions and fiber functions.
Sections~\ref{sec:states}--\ref{sec:poly} treat a $2$-connected series-parallel graph, one node of
its decomposition at a time. In Section~\ref{sec:states}, we define parity distances and the
states that the recursion passes up the decomposition. In Section~\ref{sec:chain}, we close a
chain by an edge, and in Section~\ref{sec:theta}, we combine parallel chains into a theta.
In Section~\ref{sec:poly}, we prove that every state is a single polynomial on a fixed domain;
this is what keeps the computation polynomial. In Section~\ref{sec:algo}, we handle cut vertices
through the block-cut tree, bound the number of operations and the bit lengths, and prove the
main result, Theorem~\ref{thm:main}; we then derive the simpler recursion for $\vol\QQ(G)$,
work through an example, and, by fixing every vertex coordinate at $x_v=\tfrac12$, compute the
volume of the cut polytope of every series-parallel graph. Section~\ref{sec:comp}, on the ratio $\vol\PP(G)/\vol\QQ(G)$, depends
on the earlier sections only through cut coordinates, Theorem~\ref{thm:chain}, and the fibers
over a single vertex of Section~\ref{sec:cutvertex}. It opens with the bound on what the
odd-cycle inequalities of one cycle can cut off (Theorem~\ref{thm:cap}) and then studies
hanging cycles, a long cycle glued to a fixed graph, large flowers, and relaxations that use
only the odd-cycle inequalities of short cycles. In Section~\ref{sec:concl}, we state open
questions.

% ---------------------------------------------------------------------------
\section{Preliminaries}\label{sec:prelim}

As in~\cite[Lemma~1]{LeeSkipper20}, each coordinate $y_{ij}$ with $ij\notin E$ contributes a unit
factor, so we regard $\PP(G)$ and $\QQ(G)$ as full-dimensional polytopes in $\R^{V\cup E}$, of
dimension $d=n+m$.

\subsection{Cut coordinates}
We put $u_i=x_i-\tfrac12$ and $z_e=x_i+x_j-2y_e$ for $e=ij$. The map $(x,y)\mapsto(x,z)$ is
affine with $|\det|=2^m$. Hence, if $\vol_z$ denotes volume in the coordinates $(x,z)$, then
\begin{equation}\label{eq:jac}
\vol X=2^{-m}\vol_z X\qquad(X\in\{\PP(G),\QQ(G)\}).
\end{equation}

\begin{lemma}\label{lem:edge}
For an edge $e=ij$, the constraints~\eqref{eq:Q} are equivalent to
$|u_i-u_j|\le z_e\le 1-|u_i+u_j|$. In particular, $|u_i|,|u_j|\le\tfrac12$\,, and the
admissible $z_e$-interval has length $1-2\max(|u_i|,|u_j|)$.
\end{lemma}
\begin{proof}
We translate each inequality:
\begin{itemize}[nosep]
\item $y_e\ge0$ becomes $z_e\le1+u_i+u_j$\,;
\item $x_i+x_j\le1+y_e$ becomes $z_e\le1-u_i-u_j$\,;
\item $y_e\le x_i$ and $y_e\le x_j$ become $z_e\ge\pm(u_i-u_j)$. \qedhere
\end{itemize}
\end{proof}

\noindent So, once the vertex coordinates $u$ are fixed, $\QQ(G)$ is a product of intervals, one
for each edge. This product structure underlies the results of Section~\ref{sec:comp}, in
particular Theorem~\ref{thm:cap} and Lemma~\ref{lem:cycmarg}.

For a cycle $C$ of $G$ and an odd subset $A\subset E(C)$, the odd-cycle inequality
$\OC(A)$ of~\cite{Padberg89} reads
\[
\sum_{i\in S_2}x_i-\sum_{i\in S_0}x_i+\sum_{e\in E(C)\setminus A}y_e-\sum_{e\in A}y_e\le\bigl\lfloor|A|/2\bigr\rfloor,
\]
where $S_0$ and $S_2$ are the sets of vertices of $C$ incident with zero and two edges of $A$,
respectively.

\begin{lemma}\label{lem:oc}
The inequality $\OC(A)$ is equivalent to
$\sum_{e\in A}(1-z_e)+\sum_{e\in E(C)\setminus A}z_e\ge1$.
\end{lemma}
\begin{proof}
We write $\deg_A(v)$ for the number of edges of $A$ at $v$. Substituting $z_e$\,, we find that
the left-hand side equals
\[
|A|+2\sum_{S_0}x_v-2\sum_{S_2}x_v+2\sum_Ay_e-2\sum_{E(C)\setminus A}y_e\,,
\]
because $\sum_{e\in A}(x_i+x_j)=\sum_v\deg_A(v)x_v$ and
$\sum_{e\in E(C)\setminus A}(x_i+x_j)=\sum_{v\in V(C)}(2-\deg_A(v))x_v$\,. Because $|A|$ is odd,
$(|A|-1)/2=\lfloor|A|/2\rfloor$, and the claim follows.
\end{proof}

Thus, with \emph{even length} $z_e$ and \emph{odd length} $1-z_e$ on each edge, a point of
$\QQ(G)$ violates an odd-cycle inequality if and only if some cycle has an odd labelling, that is,
one with an odd number of odd edges, of total length less than $1$. (These lengths are real
numbers; when we speak of a cycle of length $\ell$, we mean, as usual, one with $\ell$ edges.)

\begin{lemma}[switching]\label{lem:switch}
For $v\in V$, define $\sigma_v$ by $u_v\mapsto-u_v$ and $z_{vw}\mapsto1-z_{vw}$ for all
$vw\in E$, with all other coordinates fixed. Then:
\begin{itemize}[nosep]
\item $\sigma_v$ is a volume-preserving affine involution;
\item $\sigma_v$ maps $\QQ(G)$ onto $\QQ(G)$ and $\PP(G)$ onto $\PP(G)$;
\item $\sigma_v$ maps $\OC(A)$ onto $\OC(A\triangle\delta_C(v))$ for every cycle $C$ through
$v$.
\end{itemize}
\end{lemma}
\begin{proof}
In the original coordinates, $\sigma_v$ is the map $x_v\mapsto1-x_v$ and
$y_{vw}\mapsto x_w-y_{vw}$\,. This map preserves integrality and the relations $y_{vw}=x_vx_w$\,,
so it preserves $\PP(G)$. The remaining claims follow immediately from Lemmas~\ref{lem:edge}
and~\ref{lem:oc}.
\end{proof}

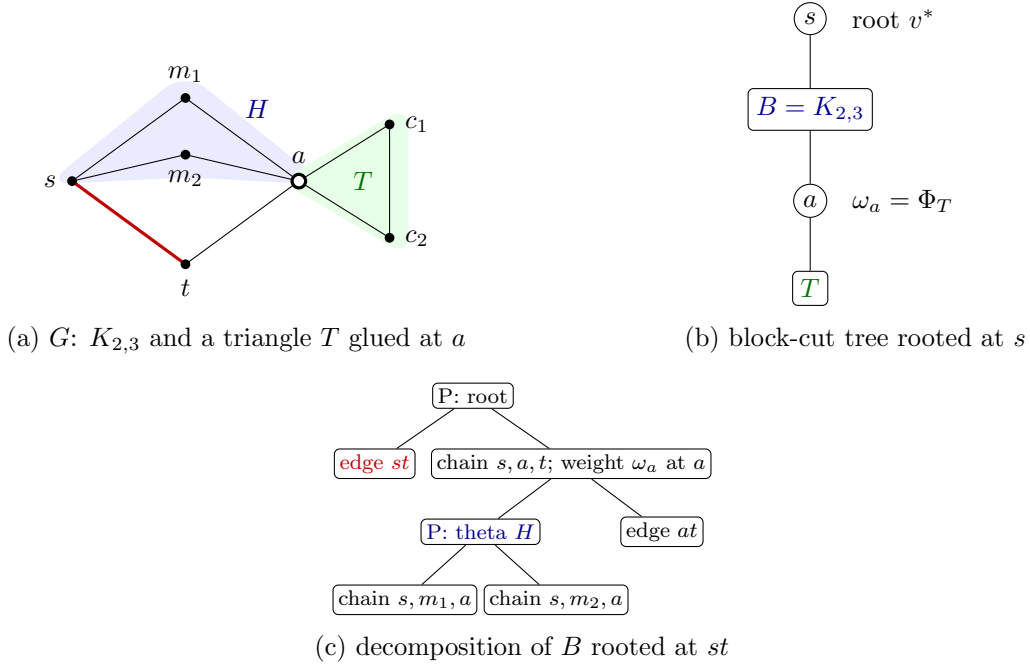
\begin{figure}[ht]
\centering
\tikzset{
  v/.style={circle,fill,inner sep=1.3pt},
  cut/.style={circle,draw,very thick,fill=white,inner sep=1.8pt},
  lab/.style={font=\small},
  root/.style={very thick,red!75!black},
  odd/.style={very thick,red!75!black,dashed},
  seg/.style={line width=2.2pt,blue!55!black},
  tn/.style={draw,rounded corners=2pt,inner sep=2pt,font=\scriptsize,fill=white},
  bn/.style={draw,rounded corners=2pt,inner sep=3pt,font=\small,fill=white}
}
\begin{minipage}[b]{0.52\textwidth}\centering
\begin{tikzpicture}[scale=1.0]
  \coordinate (S) at (0,0); \coordinate (A) at (3,0);
  \coordinate (M1) at (1.5,1.1); \coordinate (M2) at (1.5,0.35); \coordinate (T) at (1.5,-1.1);
  \coordinate (C1) at (4.2,0.75); \coordinate (C2) at (4.2,-0.75);
  \begin{scope}[on background layer]
    \fill[blue!8,rounded corners=10pt] (-0.35,-0.05) -- (1.5,1.45) -- (3.35,-0.05) -- (1.5,0.05) -- cycle;
    \fill[green!10,rounded corners=8pt] (2.75,0) -- (4.45,1.0) -- (4.45,-1.0) -- cycle;
  \end{scope}
  \draw (S) -- (M1) -- (A); \draw (S) -- (M2) -- (A); \draw (A) -- (T);
  \draw (A) -- (C1) -- (C2) -- (A);
  \draw[root] (S) -- (T);
  \foreach \p/\n/\pos in {S/s/left,M1/m_1/above,M2/m_2/below,T/t/below,C1/c_1/right,C2/c_2/right}
    \node[v,label={[lab]\pos:$\n$}] at (\p) {};
  \node[cut,label={[lab]above:$a$}] at (A) {};
  \node[lab,blue!55!black] at (2.45,0.95) {$H$};
  \node[lab,green!45!black] at (3.85,0) {$T$};
\end{tikzpicture}\\[2pt]
{\small (a) $G$: $K_{2,3}$ and a triangle $T$ glued at $a$}
\end{minipage}\hfill
\begin{minipage}[b]{0.44\textwidth}\centering
\begin{tikzpicture}[node distance=7mm]
  \node[bn,circle,inner sep=2pt] (rho) {$s$};
  \node[bn,below=of rho,text=blue!55!black] (B) {$B=K_{2,3}$};
  \node[bn,circle,inner sep=2pt,below=of B] (a) {$a$};
  \node[bn,below=of a,text=green!45!black] (T) {$T$};
  \draw (rho) -- (B) -- (a) -- (T);
  \node[lab,right=2mm of rho] {root $v^\ast$};
  \node[lab,right=2mm of a,align=left] {$\omega_a=\Phi_T$};
\end{tikzpicture}\\[2pt]
{\small (b) block-cut tree rooted at $s$}
\end{minipage}

\vspace{10pt}
\begin{minipage}[b]{0.56\textwidth}\centering
\begin{tikzpicture}[level distance=9mm,
    level 1/.style={sibling distance=26mm},
    level 2/.style={sibling distance=24mm},
    level 3/.style={sibling distance=20mm}]
  \node[tn] {P: root}
    child {node[tn,text=red!75!black] {edge $st$}}
    child {node[tn] {chain $s,a,t$; weight $\omega_a$ at $a$}
      child {node[tn,text=blue!55!black] {P: theta $H$}
        child {node[tn] {chain $s,m_1,a$}}
        child {node[tn] {chain $s,m_2,a$}}}
      child {node[tn] {edge $at$}}};
\end{tikzpicture}\\[2pt]
{\small (c) decomposition of $B$ rooted at $st$}
\end{minipage}\hfill
\caption{The running example $G$: $K_{2,3}$ with a triangle $T$ glued at the cut vertex $a$
(Section~\ref{sec:sprunning}). The open circle marks the cut vertex, and the shading marks the
two blocks. Section~\ref{sec:example} carries the example through the algorithm.}
\label{fig:running}
\end{figure}

\subsection{Series-parallel graphs}\label{sec:sprunning}
A connected graph is \emph{series-parallel} if it has no $K_4$ minor. A \emph{block} of a graph is
a maximal subgraph without a cut vertex; every block of a series-parallel graph is a single edge or
a 2-connected series-parallel graph, and the blocks meet only at cut vertices. The
\emph{block-cut tree} of a connected graph has a node for each block and for each cut vertex, and
joins each cut vertex to the blocks that contain it.

A \emph{two-terminal} graph is a graph with two distinguished vertices $s$ and $t$, its
\emph{terminals}. Two-terminal series-parallel graphs are built from single edges by two
operations: a \emph{series} composition identifies the terminal $t$ of one graph with the terminal
$s$ of the next, and a \emph{parallel} composition identifies the two terminals of several graphs
with each other. For any edge $st$ of a 2-connected series-parallel graph $B$, the graph $B$ arises
in this way with terminals $s$ and $t$, and a decomposition can be computed in linear
time; see~\cite{VTL82}.

We use the decomposition in the following canonical form, in which the two operations alternate:
\begin{itemize}[nosep]
\item a \emph{P-node} is a parallel composition. Its children are \emph{chains} and, in addition,
at most one copy of the edge $st$ joining its terminals.
\item a \emph{chain} is a series composition $s=v_0,v_1,\dots,v_k=t$ with $k\ge2$. Its
children, called \emph{segments}, are single edges and P-nodes; the vertices
$v_1,\dots,v_{k-1}$ are its \emph{internal} vertices.
\end{itemize}
A P-node is \emph{edge-closed} if it contains the edge $st$ joining its terminals, and a
\emph{theta} otherwise. The name follows the usual theta graph, three internally disjoint paths
between the same two vertices, but we also allow two chains. So the simplest theta is a
$4$-cycle with its terminals opposite; with three chains, the simplest is $K_{2,3}$\,, a theta
graph in the usual sense. The distinction matters
throughout, because in an edge-closed piece the edge $st$ pins down both parity distances
(Lemma~\ref{lem:parity}(c)), whereas in a theta nothing does.

\emph{A running example.} Figure~\ref{fig:running} shows the graph $G$ that we use as an example
in this paper. Its two sides are $\{s,a\}$ and $\{m_1,m_2,t\}$, forming $K_{2,3}$\,, and a triangle
$T$ with vertices $a$, $c_1$\,, and $c_2$ is glued at $a$; here $n=7$, $m=9$, and $d=16$. The graph
has two blocks, $B=K_{2,3}$ and $T$, which meet at the cut vertex $a$, so its block-cut tree is a
path (Figure~\ref{fig:running}(b)); we root it at $v^\ast=s$. Rooting the decomposition of $B$ at
the edge $st$ (Figure~\ref{fig:running}(c)) gives a root P-node with two children: the edge $st$, so this P-node is edge-closed;
and the chain $s,a,t$, whose internal vertex is $a$. That chain has two segments: the edge $at$,
and the theta $H$ between $s$ and $a$, whose own chains are the paths $s,m_1,a$ and $s,m_2,a$. We
carry this example through the algorithm in Section~\ref{sec:example}.

\begin{theorem}[see~\cite{Padberg89}]\label{thm:padberg}
If $G$ is series-parallel, then $\PP(G)$ is the set of points of $\QQ(G)$ that satisfy all
odd-cycle inequalities.
\end{theorem}

\noindent The odd-cycle inequalities of a cycle with a chord are implied by those of the two
shorter cycles through the chord (proof of Proposition~\ref{prop:chordless}), so the chordless
cycles suffice.

\subsection{Fiber functions}
Let $K\subset\R^N$ be a polytope, and let $\pi\colon\R^N\to\R^r$ be a linear map with
$\dim\pi(K)=r$. The \emph{fiber function} $y\mapsto\vol_{N-r}(K\cap\pi^{-1}(y))$ is the
density of the push-forward of Lebesgue measure on $K$: its integral over a set $Y\subset\R^r$ is
$\vol(K\cap\pi^{-1}(Y))$. In probabilistic language, it is $\vol K$ times the probability density
of $\pi(X)$ for a point $X$ drawn uniformly from $K$.

\begin{lemma}\label{lem:pushforward}
The fiber function is a polynomial of degree at most $N-r$ on each connected component of the
complement, in $\pi(K)$, of the union of all hyperplanes spanned by $r$ affinely independent
points of $\pi(\operatorname{vert}K)$.
\end{lemma}
\begin{proof}
We triangulate $K$ using only its vertices. For each simplex, the push-forward is a
multivariate B-spline, also called a simplex spline; see~\cite{deBoor76,Micchelli80}. Such a function is a
polynomial of degree at most $N-r$ off the hyperplanes spanned by $r$-subsets of its knots,
and the knots are the images of the vertices of the simplex. Summing over the simplices, we
obtain the claim.
\end{proof}

% ---------------------------------------------------------------------------
\section{Parity distances and states}\label{sec:states}

Let $H$ be a two-terminal piece with terminals $s$ and $t$. We call a point of $\QQ(H)$ a
\emph{configuration} of $H$, and we call it \emph{feasible} if it satisfies all odd-cycle
inequalities of $H$. For a feasible configuration, we define
\begin{itemize}[nosep]
\item $d_0(H)$: the minimum length of an $s$--$t$ path with an even labelling, that is, an even
number of odd edges;
\item $d_1(H)$: the same, with an odd number of odd edges.
\end{itemize}
Throughout, we write
\[
\clo=|u_s-u_t|,\qquad \chii=1-|u_s+u_t|,\qquad
\Delta=\{(u_s,u_t,z):\clo\le z\le\chii\}.
\]
The set $\Delta$ is a tetrahedron: it is the image of $\PP(K_2)$, and its vertices are the
images of the four $0/1$ points.

\begin{lemma}\label{lem:parity}
For every feasible configuration, the following hold.
\begin{enumerate}[label=\textup{(\alph*)},nosep]
\item $d_0(H)\ge\clo$ and $d_1(H)\ge|u_s+u_t|=1-\chii$\,.
\item $d_0(H)+d_1(H)\ge1$.
\item If $st\in E(H)$, then $d_0(H)=z_{st}$ and $d_1(H)=1-z_{st}$\,.
\end{enumerate}
\end{lemma}
\begin{proof}
(a) By Lemma~\ref{lem:edge}, $|u_i-\epsilon_eu_j|\le w_e$ for an edge $e=ij$ of length
$w_e$\,, where $\epsilon_e=+1$ if $e$ is labelled even and $\epsilon_e=-1$ if it is labelled odd.
We chain these inequalities along the path, using the triangle inequality.

(b) A shortest even path and a shortest odd path together form a closed walk with an odd number
of odd traversals. Such a walk decomposes into cycles and edges traversed twice. If some cycle
carries an odd number of odd traversals, then, by feasibility and Lemma~\ref{lem:oc}, its length
is at least $1$. Otherwise, the walk traverses some edge once with each label, for a length of
$z_e+(1-z_e)=1$.

(c) The edge $st$ yields $d_0\le z_{st}$ and $d_1\le1-z_{st}$\,. For any other $s$--$t$ path
$\pi$, the cycle $\pi+st$ yields $d_0(\pi)+(1-z_{st})\ge1$ and $d_1(\pi)+z_{st}\ge1$.
\end{proof}

\begin{definition}[states]\label{def:states}
We fix $(u_s,u_t)$ and consider the fiber of feasible configurations of $H$ over it, measured in
cut coordinates. Vertex weights (Section~\ref{sec:cutvertex}) may multiply this measure.
\begin{itemize}[nosep]
\item If $H$ is edge-closed, then $g_H(u_s,u_t,z)$ is the fiber function over
$(u_s,u_t,z_{st})\in\Delta$.
\item In general, the \emph{mass} of $H$ is the fiber function $G_H(u_s,u_t)$ over
$(u_s,u_t)$.
\item The \emph{densities} $m_0^H(u_s,u_t,w)$ and $m_1^H(u_s,u_t,w)$ are the densities of
$d_0(H)$ and $d_1(H)$ at $w$.
\end{itemize}
The \emph{state} of $H$ is what the recursion passes up: the function $g_H$ if $H$ is
edge-closed, and the mass $G_H$ together with the density $m_0^H$ otherwise.
\end{definition}

Here, as throughout, a \emph{density} is taken with respect to Lebesgue measure:
$m_0^H(u_s,u_t,w)$ is the derivative in $w$ of the fiber measure of the configurations with
$d_0(H)\le w$. These derivatives exist because $d_0$ and $d_1$ are piecewise-linear, nonconstant
functions of the configuration, so their level sets lie on finitely many hyperplanes. We do not
normalize the fiber measure to a probability measure; its total mass is $G_H$\,.

Switching at $s$ (Lemma~\ref{lem:switch}) exchanges $d_0$ and $d_1$\,, so
\begin{equation}\label{eq:m1}
m_1^H(u_s,u_t,w)=m_0^H(-u_s,u_t,w),
\end{equation}
and $G_H$ is even in each of $u_s$ and $u_t$\,. For an edge-closed piece, from
Lemma~\ref{lem:parity}(c), we have $m_0^H=g_H$ and $m_1^H(\cdot,w)=g_H(\cdot,1-w)$. For a single
edge, $g\equiv1$ on $\Delta$.

% ---------------------------------------------------------------------------
\section{Closing a chain}\label{sec:chain}

Let $\sigma_1,\dots,\sigma_k$ ($k\ge2$) be the segments of a chain $s=v_0,\dots,v_k=t$, and
suppose that we add a new edge $st$ with value $z=z_{st}$\,. For fixed $(u_s,u_t,z)\in\Delta$, the
configurations of the segments are independent given the values $u_{v_1},\dots,u_{v_{k-1}}$\,.
The cycles through $st$ consist of $st$ together with one $s$--$t$ path inside each segment.

An \emph{odd labelling} of the closed chain is a set $A\subset\{\sigma_1,\dots,\sigma_k,st\}$ of
odd cardinality, whose members are labelled odd and the others even. We put
\[
w_j=\begin{cases}d_0(\sigma_j),&\sigma_j\notin A,\\ d_1(\sigma_j),&\sigma_j\in A,\end{cases}
\qquad
w_{st}=\begin{cases}z,&st\notin A,\\1-z,&st\in A.\end{cases}
\]
The \emph{cut region} of the odd labelling $A$ is $C_A=\{w_1+\dots+w_k+w_{st}<1\}$. By Lemma~\ref{lem:oc}, the
closed piece is feasible exactly off $\bigcup_AC_A$\,.

\begin{theorem}\label{thm:chain}
The following hold.
\begin{enumerate}[label=\textup{(\alph*)},nosep]
\item The cut regions $C_A$ are pairwise disjoint.
\item On $C_A$\,, put $s_e=-1$ for $e\in A$ and $s_e=+1$ otherwise, and set
$\tilde u_j=\bigl(\prod_{i\le j}s_{\sigma_i}\bigr)u_{v_j}$\,. Then the inequalities
$|\tilde u_j-\tilde u_{j-1}|\le w_j$ for $1\le j\le k$, together with $\sum_jw_j<1-w_{st}$\,,
describe $C_A$ exactly; they imply all other constraints.
\item On $C_A$\,, every segment labelled even has $w_j\in[\clo,\chii)$, and every segment labelled odd
has $1-w_j\in(\clo,\chii]$, where $\clo$ and $\chii$ refer to the endpoints of that segment.
\end{enumerate}
Consequently, the new state is
\begin{equation}\label{eq:close}
g(u_s,u_t,z)=Q_{\rm fib}(u_s,u_t)-\sum_{A}\vol C_A(u_s,u_t,z),
\end{equation}
where $Q_{\rm fib}$ is the fiber over $(u_s,u_t)$ of the configurations in which every segment
is feasible. The masses $G_{\sigma_j}$ alone determine $Q_{\rm fib}$\,.
\end{theorem}

\noindent From part (b), for a cycle with fixed vertex coordinates, each cut region is a simplex in
the lengths $w_j$\,; using this, we obtain in Theorem~\ref{thm:cap} a bound on the volume that the
odd-cycle inequalities of a cycle can cut off in any graph.
\begin{proof}
(a) Suppose that a point lies in $C_A\cap C_B$ with $A\neq B$. We add the two strict
inequalities. Wherever $A$ and $B$ disagree, the segment or the edge $st$ contributes
$d_0+d_1\ge1$, by Lemma~\ref{lem:parity}(b), or $z+(1-z)=1$. Because both sets are odd, they
disagree in at least two places. So the left-hand side is at least $2$, a contradiction.

(b) The listed lower bounds hold by Lemma~\ref{lem:parity}(a). We chain them along the cycle
with the triangle inequality, as in the proof of that lemma, and obtain the following for each
segment $\sigma_j$ between $v_{j-1}$ and $v_j$\,. If $\sigma_j$ is labelled even, then
$w_j+|u_{v_{j-1}}+u_{v_j}|\le\sum_iw_i+w_{st}<1$, because the rest of the cycle carries an odd
number of odd traversals. If $\sigma_j$ is labelled odd, then, likewise,
$w_j+|u_{v_{j-1}}-u_{v_j}|<1$. Going once around the whole cycle, we obtain $2|u_v|<1$ for every
vertex $v$. These are all of the remaining constraints.

(c) This is a restatement of the inequalities in (b).
\end{proof}

\subsection{Rotated coordinates}
Let $S_j=w_1+\dots+w_j$ with $S_0=0$, and define $p_j=(a_j,b_j)=(S_j+\tilde u_j,S_j-\tilde u_j)$.
The cone condition $S_j-S_{j-1}\ge|\tilde u_j-\tilde u_{j-1}|$ is equivalent to
$a_{j-1}\le a_j$ and $b_{j-1}\le b_j$\,, and $d\tilde u\,dS=\tfrac12\,da\,db$. The start is
$p_0=(u_s,-u_s)$. The end $\tilde u_k=e:=\bigl(\prod_{i\le k}s_{\sigma_i}\bigr)u_t$ is fixed,
and $S_k$ ranges over $[|u_s-e|,1-w_{st})$; the lower limit expresses $p_k\ge p_0$\,.

\begin{proposition}\label{prop:box}
Let $W_j(p_{j-1},p_j)$ be the density of segment $j$:
\begin{itemize}[nosep]
\item for an edge, $W_j=1$;
\item for a P-node, $W_j=m^{\sigma_j}_{\epsilon}(u_{v_{j-1}},u_{v_j},w_j)$, where $\epsilon$
is the parity used, $u_v=\pm\tilde u$, and $w_j=S_j-S_{j-1}$\,.
\end{itemize}
Let $\omega_j$ be the (even) weight of the internal vertex $v_j$\,. We define
\[
F_1=W_1\,,\qquad
F_j(p_j)=\frac12\iint\limits_{p_0\le p_{j-1}\le p_j}F_{j-1}(p_{j-1})\,\omega_{j-1}(\tilde u_{j-1})\,W_j(p_{j-1},p_j)\,dp_{j-1}
\]
for $2\le j\le k$, where $\tilde u_{j-1}=(a_{j-1}-b_{j-1})/2$. Then
\[
\vol C_A=\int_{|u_s-e|}^{1-w_{st}}F_k(S+e,S-e)\,dS.
\]
If the segment densities are polynomials on the ranges of Theorem~\ref{thm:chain}(c), then
every $F_j$ is a polynomial in $(p_j,u_s,u_t)$.
\end{proposition}
\begin{proof}
The formula follows from Theorem~\ref{thm:chain}(b) and the change of variables. By
Theorem~\ref{thm:chain}(c), we evaluate each density inside the range where it is a polynomial.
The integration limits are affine, so the integrals preserve polynomiality.
\end{proof}

\noindent When every segment is an edge and there are no vertex weights, $W_j=1$ and the
recursion reads $F_j=\tfrac12\,F_{j-1}\star\mathbf 1_{\R^2_{\ge0}}$\,, where
$(f\star g)(p)=\int f(q)\,g(p-q)\,dq$ and $F_{j-1}$ is extended by $0$ outside $\{q\ge p_0\}$: each step convolves with the indicator of the positive
quadrant, because $p_{j-1}\le p_j$ is the only constraint that couples consecutive segments. A
piece contributes its density $W_j$ in place of the indicator, and a vertex weight multiplies
$F_{j-1}$ before the convolution.

\begin{lemma}[switching along a chain]\label{lem:chainswitch}
The volume $\vol C_A$ depends only on whether $st\in A$. In particular,
\[
\sum_{A\ni st}\vol C_A=2^{k-1}\vol C_{\{st\}}\,,\qquad \sum_{A\not\ni st}\vol C_A=2^{k-1}\vol C_{\{\sigma_1\}}\,.
\]
\end{lemma}
\begin{proof}
Switching at an internal vertex $v_j$ (Lemma~\ref{lem:switch}, applied inside the incident
segments) swaps $d_0$ and $d_1$ of $\sigma_j$ and of $\sigma_{j+1}$\,, and it fixes $u_s$\,,
$u_t$\,, and $z$. Vertex weights are even, so switching preserves the weighted measure. Hence
switching at $v_j$ maps $C_A$ bijectively onto $C_{A\triangle\{\sigma_j,\sigma_{j+1}\}}$ with the
same volume. These moves generate all changes of the labels of the segments that preserve the
parity of the number of odd segments, and whether $st\in A$ determines that parity.
\end{proof}

\begin{proposition}[$\QQ$-fiber]\label{prop:qfib}
Let $v=|u|$. The masses and the vertex weights are even, so they are functions of the variables
$v$. We write $G_{\sigma_j}(v,v')$ for the mass of segment $j$; for an edge, $G_{\sigma_j}=1-2\max(v,v')$. Then
the transfer
\[
q_1(v)=2\,\omega_1(v)G_{\sigma_1}(v_s,v),\qquad
q_j(v)=2\,\omega_j(v)\int_0^{1/2}q_{j-1}(v')G_{\sigma_j}(v',v)\,dv',
\]
followed by $Q_{\rm fib}=\int_0^{1/2}q_{k-1}(v')G_{\sigma_k}(v',v_t)\,dv'$, yields
$Q_{\rm fib}(u_s,u_t)$. Suppose that each mass is a polynomial on $v\ge v'$ and on $v\le v'$
(Theorem~\ref{thm:allpoly}(c)). Then each $q_j$ is a polynomial on $v\le v_s$ and on
$v\ge v_s$\,, and $Q_{\rm fib}$ is a polynomial on $v_t\le v_s$ and on $v_t\ge v_s$\,.
\end{proposition}
\begin{proof}
We integrate out the segments given the internal vertex values, and we fold each
$u_{v_j}\in[-\tfrac12,\tfrac12]$ onto $v\in[0,\tfrac12]$, which produces the factor $2$. The
integrals split at $v'=v$ and at $v'=v_s$\,.
\end{proof}

\noindent In operator language, each step applies the integral operator with kernel $G_{\sigma_j}$ and then
multiplies by the vertex weight, so the transfer is a product of integral operators, a
continuous version of multiplying the transfer matrices of the segments.

Suppose that a P-node contains $st$ and several chains. Then the conditions for $C_i+st$ and
$C_j+st$ imply the condition for each new cycle through two chains $C_i$ and $C_j$\,: from
Lemma~\ref{lem:parity}(c), we have $d_0(C_i)+d_1(C_j)\ge z+(1-z)$. Therefore, the states of the
closed chains multiply.

% ---------------------------------------------------------------------------
\section{Thetas}\label{sec:theta}

\begin{theorem}\label{thm:theta}
Let $H$ be a theta with chains $C_1,\dots,C_r$ ($r\ge2$). For each $i$, let
\begin{itemize}[nosep]
\item $S_i(u_s,u_t,z)$ be the state of the edge-closed piece $C_i+st$, where $st$ is a
\emph{virtual} edge;
\item $f_i(u_s,u_t,\cdot)$ be the density of $\ell_i:=1-d_1(C_i)$.
\end{itemize}
Then
\[
G_H=\prod_jS_j(\clo)+\int_{\clo}^{\chii}\sum_if_i(m)\prod_{j\ne i}S_j(m)\,dm,
\]
and, for $\clo<w<\chii$\,,
\[
m_0^H(w)=\sum_if_i(w)\prod_{j\ne i}S_j(w)-\frac{\partial}{\partial w}\prod_jS_j(w).
\]
\end{theorem}
\begin{proof}
We fix $(u_s,u_t)$; the chains are then independent. We put $r_i=d_0(C_i)$. By
Lemma~\ref{lem:parity}, $\ell_i\le r_i$\,, $r_i\ge\clo$\,, and $\ell_i\le\chii$\,. The new
cycles are the unions $C_i\cup C_j$ with $i\ne j$. Their odd-cycle conditions state that
$\ell_i\le r_j$ for all $i\ne j$; because $\ell_i\le r_i$\,, these conditions are equivalent to
$M:=\max_i\ell_i\le\min_jr_j$\,.

By Lemma~\ref{lem:parity}(c), applied to the virtual edge, the state of $C_j+st$ at $z$ is
$S_j(z)=\mu_j(\ell_j\le z\le r_j)$. For $q\in[\clo,\chii]$, we split on the value of $M$:
\begin{itemize}[nosep]
\item if $M\le q$, then the event requires $\ell_j\le q\le r_j$ for all $j$, and it has
probability $\prod_jS_j(q)$;
\item if $M=m>q$, then, off a set of measure zero, a unique $i$ attains the maximum, because
$\ell_i$ and $\ell_{i'}$ depend on disjoint sets of coordinates and their level sets have measure
zero (as after Definition~\ref{def:states}); moreover, $r_i\ge m$ holds automatically because
$\ell_i\le r_i$\,.
\end{itemize}
Hence
\[
\mu\bigl(\text{feasible},\ \min_jr_j\ge q\bigr)=\prod_jS_j(q)+\int_q^{\chii}\sum_if_i(m)\prod_{j\ne i}S_j(m)\,dm.
\]
Setting $q=\clo$\,, where the condition $\min_jr_j\ge q$ holds automatically, we obtain $G_H$\,.
Differentiating in $q$, we obtain $m_0^H$\,, because $d_0(H)=\min_jr_j$\,.
\end{proof}

By Theorem~\ref{thm:chain}(c), a parent uses only densities on $(\clo,\chii)$. From
Theorem~\ref{thm:theta}, we therefore obtain everything that the dynamic program needs.

% ---------------------------------------------------------------------------
\section{Every state is a single polynomial}\label{sec:poly}

\begin{lemma}[chain densities]\label{lem:chaindens}
Let $C$ be a chain whose segments have densities that are polynomials on the ranges of
Theorem~\ref{thm:chain}(c). Then, for $q\in(\clo,\chii)$, the following hold.
\begin{itemize}[nosep]
\item The density of $d_0(C)$ at $q$ is $2^{k-1}F^{\{st\}}_k(q+u_t,q-u_t)$.
\item The density of $1-d_1(C)$ at $q$ is $2^{k-1}F^{\{\sigma_1\}}_k(1-q-u_t,1-q+u_t)$.
\end{itemize}
Both are single polynomials in $(u_s,u_t,q)$.
\end{lemma}
\begin{proof}
We add a virtual edge $st$ with value $q\in[\clo,\chii]$; by Lemma~\ref{lem:edge}, this value is
admissible. For $A\ni st$, the cut region $C_A$ is the event that the labelling of $C$ with odd set
$A\setminus\{st\}$, which has an even number of odd edges, has length less than $q$. By Theorem~\ref{thm:chain}(a), these events
are disjoint, so $\mu\{d_0(C)<q\}=\sum_{A\ni st}\vol C_A$\,.

For $A\ni st$, the chain part of $A$ has even size, so $e=+u_t$ for every such $A$. The
integration range in Proposition~\ref{prop:box} is then $[\,|u_s-u_t|,\ q)=[\clo,q)$. Hence
\[
\mu\{d_0(C)<q\}=\sum_{A\ni st}\int_{\clo}^{q}F^A_k(S+u_t,S-u_t)\,dS,
\]
and differentiating in $q$ removes the only nonpolynomial ingredient, the lower limit
$\clo=|u_s-u_t|$. We then apply Lemma~\ref{lem:chainswitch}.

For $1-d_1$\,, we argue in the same way with $A\not\ni st$. Then $e=-u_t$\,, the range is
$[\,|u_s+u_t|,1-q)$, and we differentiate the upper limit.
\end{proof}

\begin{theorem}\label{thm:allpoly}
For every node of the canonical decomposition of a 2-connected series-parallel graph, the
following hold.
\begin{enumerate}[label=\textup{(\alph*)},nosep]
\item The state $g$ of an edge-closed node is a single polynomial on $\Delta$.
\item The density $m_0$ of a theta is a single polynomial on
$\{(u_s,u_t,w):\clo<w<\chii\}$.
\item The mass $G$ of any node is a polynomial on each of the four triangles into which the
lines $u_s=\pm u_t$ divide $[-\tfrac12,\tfrac12]^2$. Equivalently, as a function of
$(|u_s|,|u_t|)$, it is a polynomial on $|u_s|\ge|u_t|$ and on $|u_s|\le|u_t|$.
\end{enumerate}
The same conclusions hold with even polynomial vertex weights.
\end{theorem}
\begin{proof}
(a) Without weights, $g$ is the fiber function of the $0/1$ polytope $\PP(H)$ under the linear
map to $(x_s,x_t,z_{st})$. The images of the vertices are the four vertices of $\Delta$, so, by
Lemma~\ref{lem:pushforward}, the only possible walls are the facets of $\Delta$. With weights,
we apply the same argument to the polytope obtained by attaching to each weighted vertex the
series-parallel subgraph that generated its weight.

(c) The same argument applies to the map to $(x_s,x_t)$. The images are the corners of the
square, so the possible walls are its sides and diagonals. On the triangle
$\{u_s\ge|u_t|\}$, switching at $t$ implies that the polynomial is even in $u_t$\,. The other
triangles are similar.

(b) We argue by induction over the decomposition. For an edge-closed segment, $m_0=g$ is a
polynomial by (a). For a theta, the functions $S_j$ in Theorem~\ref{thm:theta} are states of
edge-closed pieces, hence polynomials by (a). The functions $f_i$ are densities of
$1-d_1(C_i)$, hence polynomials by Lemma~\ref{lem:chaindens} and the induction hypothesis.
Therefore, the formula of Theorem~\ref{thm:theta} is a single polynomial.
\end{proof}

% ---------------------------------------------------------------------------
\section{The algorithm}\label{sec:algo}

\subsection{2-connected graphs}
We evaluate the canonical decomposition from the bottom up.
\begin{itemize}[nosep]
\item We close each chain in an edge-closed P-node via~\eqref{eq:close}, using
Proposition~\ref{prop:qfib} for $Q_{\rm fib}$ and using Proposition~\ref{prop:box} with
Lemma~\ref{lem:chainswitch} for the two representative cut regions. The results for the
different chains multiply.
\item We evaluate each theta by Theorem~\ref{thm:theta}; we close its chains by a virtual
edge, and we obtain the functions $f_i$ from Lemma~\ref{lem:chaindens}.
\end{itemize}
We root the decomposition at an edge $st$. The root is then edge-closed, and
$\vol_z\PP(B)=\int_\Delta g\,du_s\,du_t\,dz$.

\subsection{Cut vertices}\label{sec:cutvertex}
Let $G$ be connected. We root its block-cut tree at a vertex $v^\ast$. For a block $B$ with
parent vertex $v$ and child vertices $c$, we define the \emph{vertex weight} $\omega_c$ as the
product of the marginals $\Phi$ of the blocks below $c$. The marginal of $B$ is
\[
\Phi_B(u_v)=\int g_B(u_v,u_w,z)\,\omega_w(u_w)\,du_w\,dz,
\]
where $w$ is a neighbor of $v$ in $B$, the decomposition of $B$ is rooted at $vw$, and every
vertex of $B$ other than $v$ and $w$ contributes its weight as an internal chain vertex
(Propositions~\ref{prop:box} and~\ref{prop:qfib}). Then
\[
\vol\PP(G)=2^{-m}\int_{-1/2}^{1/2}\prod_{B\ni v^\ast}\Phi_B(u)\,du.
\]
Each $\Phi_B$ is the fiber function of a $0/1$ polytope over a single coordinate, so, by
Lemma~\ref{lem:pushforward}, it is a single polynomial on $[-\tfrac12,\tfrac12]$; by switching,
it is even. The weights are even, so Lemma~\ref{lem:chainswitch} still applies.

If $B$ is a bridge $vw$, no decomposition is needed. Then $g_B\equiv1$ on $\Delta$, and
integrating out $z$ leaves the edge mass $1-2\max(|u_v|,|u_w|)$ of Lemma~\ref{lem:edge}. We fold
$u_w$ onto $[0,\tfrac12]$ and put $\Omega_w(t)=\int_0^t\omega_w(s)\,ds$. Because
$\tfrac{d}{dv}\,2\int_0^{1/2}(1-2\max(v,t))\,\omega_w(t)\,dt=-4\,\Omega_w(v)$ and the left-hand
side vanishes at $v=\tfrac12$, we obtain
\begin{equation}\label{eq:bridge}
\Phi_B(u_v)=4\int_{|u_v|}^{1/2}\Omega_w(t)\,dt .
\end{equation}
Because $\omega_w$ is even, $\Omega_w$ is odd, and~\eqref{eq:bridge} is an even polynomial in
$u_v$\,. A leaf $w$ has $\omega_w\equiv1$ and $\Phi_B(u_v)=\tfrac12-2u_v^2$\,; for a single edge
this gives $\vol\PP(K_2)=\tfrac12\int_{-1/2}^{1/2}(\tfrac12-2u^2)\,du=\tfrac16$\,.
If $G$ is disconnected, its coordinates split along the components, and the volumes of the
components multiply; an isolated vertex contributes the factor $1$.

\subsection{Complexity}

The cost of the algorithm is governed by the degrees of the polynomials that it forms, and these
are controlled by the sizes of the subgraphs that the polynomials summarize, not merely by $d$.
Call a P-node \emph{nested} if it is a segment of a chain, and let
\[
\nu=\#\{\text{nested P-nodes}\},\qquad
\tau=\#\{\text{chains that are children of thetas}\}.
\]
For a P-node $P$ with terminals $s$ and $t$, let $|P|$ be the number of coordinates of the
subgraph spanned by $P$, together with those of the subgraphs that generate the vertex weights at
its vertices other than $s$ and $t$, but excluding $u_s$ and $u_t$\,. Let $\Pi$ be the set of
P-nodes that are nested or have at least two chain children.

\begin{lemma}\label{lem:nutau}
Let $G$ be connected and series-parallel, with $\mu=m-n+1$. Then $|\Pi|\le\mu$, $\nu\le\mu$, and
$\tau\le2\mu$. Moreover, $\Pi=\emptyset$ if and only if $G$ is a cactus, and in that case
$\nu=\tau=0$.
\end{lemma}
\begin{proof}
A P-node with $c$ chain children, and possibly the terminal edge, contributes $c$ or $c-1$
independent cycles, in either case at least one, and the cycle spaces of distinct P-nodes are
independent. Hence the total number of P-nodes is at most $\mu$, and so are $|\Pi|$ and $\nu$.

We root each block at one of its edges, so the root P-node of a block contains that edge and is
edge-closed; every other P-node, and in particular every theta, is nested. So the number of
thetas is at most $\nu$. A theta with $r$ chains contributes $r-1$ independent cycles, so
$\tau=\sum_\theta r_\theta\le\mu+\#\{\text{thetas}\}\le2\mu$.

If $G$ is a cactus, then each block is an edge or a cycle; the decomposition of a cycle rooted at
one of its edges is a single P-node whose children are that edge and one chain of edges, so
$\Pi=\emptyset$ and $\nu=\tau=0$. Conversely, if $\Pi=\emptyset$, then no P-node is nested, so the
decomposition of each 2-connected block is its root P-node alone; that P-node has a single chain,
whose segments are edges, so the block is a cycle.
\end{proof}

The condition $\nu=\tau=0$ alone does not characterize cacti. Two triangles that share an edge
$st$, decomposed with root $st$, give a root P-node with the edge $st$ and two chains of length
two; here $\nu=\tau=0$, but the root P-node lies in $\Pi$.

\begin{proposition}\label{prop:ops}
The algorithm computes $\vol\PP(G)$ with $O(d^2)$ arithmetic operations on rationals if $G$ is a
forest, and with
\[
O\Bigl(d^5+\min\Bigl\{d^7,\ d^3\sum_{P\in\Pi}|P|^4\Bigr\}\Bigr)
\]
arithmetic operations otherwise. In particular, it uses $O(d^5)$ operations when $G$ is a cactus,
and $O(d^7)$ operations in general.
\end{proposition}
\begin{proof}
We treat each component separately; multiplying their volumes costs $O(n)$.

\emph{Supports and degrees.} To every polynomial that the algorithm forms, we assign its
\emph{support}: the set of coordinates of $G$ that it summarizes, that is, the coordinates of the
corresponding subgraph together with those of the subgraphs that generated the vertex weights it
has absorbed. For a vertex weight $\omega_c$\,, a block marginal, or a function in the state of a
two-terminal piece, we exclude the coordinates of the vertices at which it is evaluated;
thus the functions in the state of a P-node $P$ have support of size $|P|$. Every polynomial
has degree at most the size of its support. For $g_H$ and $G_H$ this is the bound of
Lemma~\ref{lem:pushforward}, as in the proof of Theorem~\ref{thm:allpoly}; in general it follows
by induction over the recursion, because a product adds degrees and supports, each integration
raises the degree by one and consumes one coordinate (an internal vertex, an edge, or the
argument of a density, whose degree is one less than the size of its support), and derivatives
and substitutions do not raise the degree. A polynomial in $v$ variables whose support has $k\ge1$
elements therefore has at most $\binom{k+v}{v}\le(2k)^v$ terms. As before, a product costs at
most the product of the numbers of terms, and an antiderivative, a substitution of the limits of
Proposition~\ref{prop:box}, or a derivative costs a number of operations linear in the number of
terms.

\emph{Merging.} Every product that the algorithm forms multiplies two polynomials with disjoint
supports $X$ and $Y$, and we claim that
\begin{equation}\label{eq:merge}
\sum_{\text{products}}|X|\,|Y|\le3\binom d2.
\end{equation}
The left-hand side counts, for each product, the pairs of coordinates that it separates. Combine
the block-cut tree and the decompositions of the blocks into a single tree, in which the
subgraphs that generate the weights at the vertices of a chain hang below that chain. Two
coordinates are separated only at the node where they first meet, and there at most three times:
once in each of the three passes over a chain (the two representative cut regions of
Lemma~\ref{lem:chainswitch} and the $\QQ$-fiber transfer), in each of the three products of the
theta recursion below, and once in the remaining products.

\emph{Univariate operations.} These occur for every graph. For a bridge, \eqref{eq:bridge} costs
$O(d)$, so all bridges cost $O(n\,d)$. Each vertex weight, and the integrand at the root, is a
product of univariate block marginals, so by~\eqref{eq:merge} these products cost
$\sum4|X|\,|Y|=O(d^2)$. The final integral over $u$ costs $O(d)$. If $G$ is a forest, nothing
else occurs, so the total is $O(d^2)$.

\emph{Chain steps at edges.} Consider step $j$ of the box recursion of
Proposition~\ref{prop:box}, with $\sigma_j$ an edge, and let $X$ be the support of $F_{j-1}$ and
$Y$ that of $\omega_{j-1}$\,. If $Y\ne\emptyset$, the composition $\omega_{j-1}(\tilde u_{j-1})$
with $\tilde u_{j-1}=(a_{j-1}-b_{j-1})/2$ is a bivariate polynomial with at most $(2|Y|)^2$ terms,
which we obtain by binomial expansion in $O(|Y|^2)=O(d^2)$ operations. The product with $F_{j-1}$
costs at most $(2|X|)^3(2|Y|)^2\le32\,d^3|X|\,|Y|$, so by~\eqref{eq:merge} all these products
together cost $O(d^5)$. The compositions cost $O(n\,d^2)=O(d^3)$. Integrating out $a_{j-1}$ and
$b_{j-1}$ costs $O(d^3)$ per step, and there are at most $3m$ steps, so this costs $O(d^4)$. The
$\QQ$-fiber transfer of Proposition~\ref{prop:qfib} multiplies two-piece bivariate polynomials,
at cost $O(d^2|X|\,|Y|)$ per product, hence $O(d^4)$ by~\eqref{eq:merge}, and integrates trivariate
ones, at cost $O(d^3)$ per step, hence $O(d^4)$. For a 2-connected block, multiplying $g_B$ by
$\omega_w$ costs $O(d^2|X|\,|Y|)$, and integrating out $(u_w,z)$ costs $O(d^3)$; over all blocks
this is $O(d^4)$. Hence, if $G$ has a cycle, everything except the terms below costs $O(d^5)$.

\emph{Chain steps at P-nodes.} Now let $\sigma_j=P$ be a P-node, and let $X$ be the support of
$F_{j-1}\,\omega_{j-1}$\,. Then $W_j$ is the density $m_0$ of $P$ evaluated at $u_{v_{j-1}}$\,,
$u_{v_j}$\,, and $w_j$\,, which are affine in $(a_{j-1},b_{j-1},a_j,b_j)$, and it has at most
$(2|P|)^4$ terms. Horner's scheme, applied to one pair of variables at a time, computes this
composition in $O(|P|^5)$ operations. The product with $F_{j-1}\,\omega_{j-1}$ costs at most
$(2|X|)^3(2|P|)^4$, which is $O(d^5|X|\,|P|)$ and also $O(d^3|P|^4)$. The product has degree at
most $|P|$ in $(a_j,b_j)$, so it has $O(d^3|P|^2)$ terms, and integrating out $a_{j-1}$ and
$b_{j-1}$ costs $O(d^3|P|^2)$. Summing over the at most $3\nu$ such steps, and using~\eqref{eq:merge}
for the products and $|P|\le d$ and $\nu\le d$ for the rest, we obtain
$O\bigl(\min\{d^7,\ d^3\sum_{P\in\Pi}|P|^4\}\bigr)$.

\emph{P-nodes with several chains.} Let $P$ have chains $C_1,\dots,C_r$ with $r\ge2$. If $P$ is
edge-closed, its state is the product of the states of the closed chains. If $P$ is a theta, we
evaluate Theorem~\ref{thm:theta} through the recursion
\[
\Gamma_1=S_1\,,\quad \Sigma_1=f_1\,,\qquad \Gamma_{i+1}=\Gamma_i\,S_{i+1}\,,\quad \Sigma_{i+1}=\Sigma_i\,S_{i+1}+\Gamma_i\,f_{i+1}\,,
\]
so that $\Sigma_r=\sum_if_i\prod_{j\ne i}S_j$\,, $m_0=\Sigma_r-\partial_w\Gamma_r$\,, and
$G_P=\Gamma_r(\clo)+\int_{\clo}^{\chii}\Sigma_r(q)\,dq$. In either case each product multiplies a
trivariate polynomial whose support $X$ lies in that of $P$ by one whose support is that of a
chain $C_i$\,, at cost at most $(2|X|)^3(2|C_i|)^3$, which is $O(d^4|X|\,|C_i|)$ and also
$O(|P|^3|C_i|^3)$. By~\eqref{eq:merge}, all these products cost $O(d^6)$. At a single P-node they
cost $O(|P|^3\sum_i|C_i|^3)=O(|P|^6)=O(d^2|P|^4)$; the derivative and the integral cost
$O(|P|^3)$. Every such P-node lies in $\Pi$, so the total is
$O\bigl(\min\{d^6,\ d^3\sum_{P\in\Pi}|P|^4\}\bigr)$.

Adding the contributions proves the bound. For a cactus, $\Pi=\emptyset$ by
Lemma~\ref{lem:nutau}, which leaves $O(d^5)$; in general, the minimum is at most $d^7$.
\end{proof}

\begin{remark*}
The bound depends on the pieces in $\Pi$ through their sizes, not their number, and it
degenerates correctly at both ends. For a forest there are no cycles, hence no chains and no cut
regions: the recursion reduces to the bridge transfer~\eqref{eq:bridge} and to products of
univariate marginals, and costs $O(d^2)$. Because $\QQ(G)=\PP(G)$ for a forest, the same bound
holds for $\vol\QQ$\,. For a cactus, only the chain steps at edges remain, and they cost $O(d^5)$.
Pieces of bounded size are inexpensive even when there are many of them: take a series
composition of $k$ copies of the $4$-cycle, and close it by an edge joining its ends; each copy is
a theta with two chains of length two. Then $\nu=k$, $\mu=k+1$, and $m=4k+1$, every $P\in\Pi$ has $|P|\le6$, and the bound is
$O(d^5+k\,d^3)=O(d^5)$. The term $d^7$ arises only from a single chain step in which both the
prefix and the P-node segment have support of size $\Theta(d)$. For example, close by an edge a
chain that consists of a path followed by a theta of two paths, all three of length $\Theta(d)$.
With schoolbook multiplication, the exponent $7$ therefore cannot be improved by this accounting
alone. The table collects the worst-case bounds; the refined forms, in terms of the sizes of
the pieces in $\Pi$, are those of Propositions~\ref{prop:ops} and~\ref{prop:qcost}.

\begin{center}\small
\begin{tabular}{lcll}
\toprule
$G$ & $\Pi$ & $\vol\PP$ & $\vol\QQ$\\
\midrule
forest & $\emptyset$ & $O(d^2)$ & $O(d^2)$\\
cactus & $\emptyset$ & $O(d^5)$ & $O(d^3)$\\
series-parallel & $|\Pi|\le\mu$ & $O(d^7)$ & $O(d^4)$\\
\bottomrule
\end{tabular}
\end{center}

\noindent For $\vol\QQ$ the structure of the graph matters less. Its states are bivariate whatever
the graph, so the exponent drops by one from series-parallel graphs to cacti, rather than by two:
in a cactus every segment is an edge, and multiplying by an edge state is inexpensive; see
Proposition~\ref{prop:qcost}.
\end{remark*}

\begin{proposition}\label{prop:bits}
Every rational number that the algorithm produces has bit length $O(d^2)$.
\end{proposition}
\begin{proof}
For a polynomial $p$, let $\den(p)$ be the least positive integer such that $\den(p)\,p$ has
integer coefficients, and let $\|p\|_1$ be the sum of the absolute values of the coefficients
of $p$. It suffices to bound $\lambda(p)=\log\den(p)+\log\|p\|_1$\,.

Only the following operations occur, on polynomials of degree $D\le d$ and with affine forms
$L$ that have coefficients in $\{0,\pm\tfrac12,\pm1,\pm2\}$ and satisfy $\|L\|_1\le8$:
\begin{itemize}[nosep]
\item products, with $\den(pq)\mid\den(p)\den(q)$ and $\|pq\|_1\le\|p\|_1\|q\|_1$\,;
\item substitutions $x\mapsto L$, which multiply $\den$ by at most $2^D$ and $\|\cdot\|_1$ by at
most $8^D$;
\item antiderivatives, which multiply $\den$ by at most
$\operatorname{lcm}(1,\dots,D+1)<3^{D+1}$ (see~\cite{Hanson72}) and do not increase $\|\cdot\|_1$\,;
\item definite integrals, which are antiderivatives followed by two substitutions;
\item derivatives, and multiplication by the constants $2^{\pm1}$ and $2^{k-1}$.
\end{itemize}
Here an \emph{operation} is one of the polynomial operations just listed, not one of the
coefficient operations that Proposition~\ref{prop:ops} counts; one polynomial operation may use
many coefficient operations, and the algorithm performs $O(m)$ polynomial operations in total.
Each operation other than a product adds $O(d)$ to $\lambda$. Within a product term at a node,
each child contributes at most one factor: in the theta formula, $f_i$ and $S_i$ never occur
together. A sum of $r$ terms adds at most $\log r$. Hence
\[
\lambda(\text{node})\le\sum_{\text{children}}\lambda(\text{child})+O(d)\cdot\#\{\text{operations at the node}\}+O(\log d).
\]
Unrolling this recursion over the tree, we obtain $\lambda=O(d)\cdot O(m)=O(d^2)$ for every
intermediate object.
\end{proof}

We can now state and prove the main result.

\begin{theorem}\label{thm:main}
Let $G$ be a series-parallel graph with $n$ vertices and $m$ edges, and let $d=n+m$. Then
$\vol\PP(G)$ can be computed exactly with $O(d^7)$ arithmetic operations on rationals of bit
length $O(d^2)$, and with $O(d^5)$ arithmetic operations if $G$ is a cactus. In particular,
$\vol\PP(G)$ can be computed in time polynomial in $d$.
\end{theorem}
\begin{proof}
We decompose $G$ into blocks, and each block into its canonical series-parallel decomposition,
in linear time (see~\cite{VTL82}). We then run the algorithm above. Correctness follows from
Theorems~\ref{thm:padberg}, \ref{thm:chain}, and~\ref{thm:theta}. The bounds on
the number of operations and on bit lengths are those of Propositions~\ref{prop:ops}
and~\ref{prop:bits}; together they give a running time polynomial in $n+m$.
\end{proof}

\begin{remark}
Cactus forests, which~\cite[Section~8]{LeeSkipper20} singles out as a first step, form the simplest case. For cacti, every P-node is either a cycle closed by an edge or a
vertex joint. The recursion then reduces to one-dimensional transfers with vertex weights, and,
for cycles, it reproduces the closed forms of~\cite[Theorems~13 and~19]{LeeSkipper20}. For an
unweighted cycle, the transfer of Proposition~\ref{prop:qfib} is an integral operator whose
spectrum we compute in Lemma~\ref{lem:cycmarg}; this gives a second derivation of
$\vol\QQ(C_\ell)$.
\end{remark}

\subsection{Volumes of \texorpdfstring{$\QQ$}{Q}}\label{sec:volQ}
The relaxation $\QQ$ imposes no cycle conditions, so the dynamic program simplifies
considerably. In~\cite[Theorem~4]{LeeSkipper20}, the computation of $\vol\QQ(G)$ is reduced
to counting the linear extensions of the incidence poset of $G$, and, for graphs of bounded
treewidth, the algorithm of~\cite{KHNK16} is applied. That algorithm counts the linear extensions of an $n$-element poset
whose cover graph has treewidth $t$ with $O(n^{t+4})$ arithmetic operations: it applies
inclusion-exclusion over $k=1,\dots,n$ and, for each $k$, variable elimination at cost
$O(n^2k^{t+1})$. For a series-parallel graph, the incidence poset has $d$ elements, and its cover
graph, which is $G$ with every edge subdivided, has treewidth at most
$2$ (see~\cite[Section~2]{LeeSkipper20}); this route therefore uses $O(d^6)$ arithmetic operations.

In general, the dependence of the exponent on $t$ cannot be removed: \cite{EGKO16} showed that counting linear extensions is not fixed-parameter tractable
when parameterized by the treewidth of the cover graph, unless $\mathrm{W}[1]=\mathrm{FPT}$. Their
construction uses posets with long chains, however, so it does not bear on incidence posets,
which have height two. The polynomial-time algorithm of~\cite{Mohring89} for
series-parallel posets does not apply either: series-parallel posets are N-free, whereas the
incidence poset of a path with edges $ij$ and $jk$ already contains an induced N, namely
$y_{ij}<x_j$\,, $y_{jk}<x_j$\,, and $y_{jk}<x_k$\,, with $y_{ij}$ and $x_k$ incomparable. We give
a direct recursion.

For a two-terminal piece $H$ with terminals $s$ and $t$, let $q_H$ be the fiber function of
$\QQ(H)$ over $(u_s,u_t)$, measured in cut coordinates and multiplied by the vertex weights
(Section~\ref{sec:cutvertex}) at the vertices of $H$ other than $s$ and $t$. By switching
(Lemma~\ref{lem:switch}), $q_H$ depends only on $(v_s,v_t)=(|u_s|,|u_t|)$; for $\QQ$, the vertex
weights are likewise functions of $v=|u|$.

\begin{proposition}\label{prop:qdp}
The following hold.
\begin{enumerate}[label=\textup{(\alph*)},nosep]
\item For an edge, $q=1-2\max(v_s,v_t)$.
\item If $H$ is the series composition of $H_1$\,, with terminals $s$ and $m$, and $H_2$\,, with
terminals $m$ and $t$, then
\[
q_H(v_s,v_t)=2\int_0^{1/2}q_{H_1}(v_s,v)\,\omega_m(v)\,q_{H_2}(v,v_t)\,dv.
\]
\item If $H$ is the parallel composition of $H_1,\dots,H_r$\,, then $q_H=\prod_iq_{H_i}$\,.
\item Each $q_H$ is a polynomial on $\{v_s\ge v_t\}$ and on $\{v_s\le v_t\}$, of degree at most
$d$.
\end{enumerate}
\end{proposition}
\begin{proof}
(a) This is Lemma~\ref{lem:edge}. (b) Given $u_m$\,, the configurations of $H_1$ and $H_2$ are
independent, and folding $u_m\in[-\tfrac12,\tfrac12]$ onto $v\in[0,\tfrac12]$ produces the factor
$2$. (c) Given $(u_s,u_t)$, the constraints of $\QQ$ on the different parallel pieces involve
disjoint sets of variables. (d) We argue by induction. Suppose that $q_{H_1}$ and $q_{H_2}$ are
two-piece polynomials. On each ordering of $v_s$ and $v_t$\,, we split the integral in~(b) at
$v=v_s$ and $v=v_t$\,; the result is a sum of three integrals of polynomials whose limits lie among
$0$, $v_s$\,, $v_t$\,, and $\tfrac12$\,, hence a polynomial. Products preserve the two-piece
structure. The degree is at most the number of coordinates of $H$, and of the subgraphs that
generated its weights, other than $u_s$ and $u_t$\,.
\end{proof}

For a block $B$ with parent vertex $v$ and a neighbor $w$ of $v$ in $B$, we decompose $B$ with
terminals $v$ and $w$, as in Section~\ref{sec:cutvertex}. The marginal
\[
\Phi_B(v_v)=2\int_0^{1/2}q_B(v_v,v_w)\,\omega_w(v_w)\,dv_w
\]
is a polynomial in $v_v$ on $[0,\tfrac12]$, by splitting at $v_w=v_v$\,, and
\[
\vol\QQ(G)=2^{-m}\cdot2\int_0^{1/2}\prod_{B\ni v^\ast}\Phi_B(v)\,dv.
\]
For a bridge, $q_B$ is the edge state of~(a), and $\Phi_B$ is given by~\eqref{eq:bridge}. Within
a chain, we apply~(b) from $s$ to $t$, adding one segment at a time; at a P-node, we multiply the
states of the children one at a time. Proposition~\ref{prop:qfib} is the special case of this
recursion used inside Theorem~\ref{thm:chain}.

\begin{proposition}\label{prop:qcost}
The recursion above computes $\vol\QQ(G)$ with $O(d^2)$ arithmetic operations if $G$ is a forest,
and with
\[
O\Bigl(d^3+\min\Bigl\{d^4,\ d^2\sum_{P\in\Pi}|P|^2\Bigr\}\Bigr)
\]
arithmetic operations otherwise, on rationals of bit length $O(d^2)$. In particular, it uses
$O(d^3)$ operations when $G$ is a cactus, and $O(d^4)$ operations in general.
\end{proposition}
\begin{proof}
We use the supports and the counts of terms from the proof of Proposition~\ref{prop:ops}. By
Proposition~\ref{prop:qdp}(d), each state $q_H$ has degree at most the size of its support, and
so does each partial result in a chain, by the same induction. We count the two pieces of a
two-piece polynomial separately, which affects only the constants. For a forest, every block is
a bridge, and the recursion coincides with that of Section~\ref{sec:cutvertex}, so the bound
$O(d^2)$ of Proposition~\ref{prop:ops} applies.

\emph{Merging.} Every product again multiplies two polynomials with disjoint supports $X$ and $Y$.
In a series step, a pair of coordinates is separated at most six times: once in each product of
a piece of the partial result with the weight, and once in each of the products, on the three
subintervals of the two orderings of $v_s$ and $v_t$\,, of a piece of that product with a piece of
the new segment. Elsewhere a pair is separated at most twice. As for~\eqref{eq:merge}, we obtain
$\sum|X|\,|Y|\le6\binom d2$. A product of two polynomials in at most two variables each costs at
most $(2|X|)^2(2|Y|)^2\le16\,d\,|X|\,|Y|\min\{|X|,|Y|\}$, and it costs at most
$8\,d\,|X|\,|Y|$ if one factor is univariate.

\emph{Series steps.} We add a segment $\sigma$, with support $Y$, to a partial result $q(v_s,v)$,
with support $X$, through a vertex with weight $\omega$, with support $W$. The product
$q\,\omega$ costs $8\,d\,|X|\,|W|$, hence $O(d^3)$ in total. Let $X'=X\cup W$. The product of
$q\,\omega$ with the state of $\sigma$ costs at most $16\,d\,|X'|\,|Y|^2$. If $\sigma$ is an edge,
then $|Y|=1$, and these products cost $O(d^3)$ in total. If $\sigma$ is a nested P-node $P$, then
this cost is at most $16\,d^2|X'|\,|P|$ and at most $16\,d^2|P|^2$, so these products cost
$O\bigl(\min\{d^4,\ d^2\sum_{P\in\Pi}|P|^2\}\bigr)$ in total. The resulting trivariate polynomial
has degree at most $|X'|$ in $v_s$\,, at most $|Y|$ in $v_t$\,, and at most $|X'|+|Y|$ in $v$, so
it has at most $(|X'|+1)(|Y|+1)(|X'|+|Y|+1)\le8\,d\,|X'|\,|Y|$ terms. Integrating out $v$
therefore costs $O(d^3)$ in total.

\emph{Parallel compositions.} Let $P$ be a P-node. If $P\notin\Pi$, it has one chain and the edge
$st$, whose state has support of size $1$; the product costs $16\,d\,|X|$, hence $O(d^3)$ in
total. If $P\in\Pi$, each product multiplies a partial product, whose support lies in that of
$P$, by the state of a child $C$, at cost at most $16\,d^2|X|\,|C|$ and at most
$16\,|P|^2|C|^2$. At $P$ these products cost at most $16\,|P|^2\sum_C|C|^2\le16\,|P|^4\le
16\,d^2|P|^2$, and over all P-nodes they cost $O(d^4)$ by merging.

\emph{Block marginals and the root.} For a 2-connected block, the product $q_B\,\omega_w$ costs
$8\,d\,|X|\,|W|$, hence $O(d^3)$ in total, and integrating out $v_w$ costs $O(d^2)$ per block. The
bridges, the products of univariate marginals, and the final integral cost $O(d^2)$, as in the
proof of Proposition~\ref{prop:ops}.

Adding the contributions proves the bound. For a cactus, $\Pi=\emptyset$ by
Lemma~\ref{lem:nutau}, which leaves $O(d^3)$. The bound on bit lengths follows as in
Proposition~\ref{prop:bits}.
\end{proof}

By~\cite[Theorem~4]{LeeSkipper20}, the number of linear extensions of the incidence poset
$(\mathscr V(G),\prec)$ of $G$ equals $2^{|E|}\,d!\,\vol\QQ(G)$. We therefore obtain the following
corollary, in contrast with the theorem of~\cite{DittmerPak20} that this counting
problem is \#P-complete for arbitrary graphs.

\begin{corollary}\label{cor:linext}
For every series-parallel graph $G$, we can count the linear extensions of the incidence poset
of $G$ with $O(d^4)$ arithmetic operations, and with $O(d^3)$ operations if $G$ is a cactus.
\end{corollary}

Both volumes also have a reading in terms of cut polytopes. The \emph{suspension} $\nabla G$ of
$G$ is obtained by adding a vertex $a$ adjacent to every vertex of $G$; its cut polytope
$\mathrm{Cut}(\nabla G)$ is the convex hull of the cut vectors of $\nabla G$ in $\R^{V\cup E}$,
where the coordinate $x_{ia}$ of the edge $ia$ is indexed by $i\in V$. The \emph{rooted metric
polytope} of $K_{n+1}$ with root $a$ is defined by the four triangle inequalities
\[
x_{ij}\le x_{ia}+x_{ja}\,,\qquad x_{ia}\le x_{ij}+x_{ja}\,,\qquad x_{ja}\le x_{ij}+x_{ia}\,,\qquad
x_{ij}+x_{ia}+x_{ja}\le2
\]
for all pairs $\{i,j\}$ of vertices other than $a$; we write $\mathrm{RMet}(\nabla G)$ for its
projection onto the coordinates indexed by the edges of $\nabla G$; see~\cite{KLS97}.

\begin{corollary}\label{cor:rmet}
For every graph $G$ without isolated vertices,
\[
\vol\mathrm{Cut}(\nabla G)=2^{|E|}\vol\PP(G)\qquad\text{and}\qquad
\vol\mathrm{RMet}(\nabla G)=2^{|E|}\vol\QQ(G).
\]
Hence, if $G$ is series-parallel, then $\vol\mathrm{Cut}(\nabla G)$ can be computed with $O(d^7)$
arithmetic operations, and with $O(d^5)$ if $G$ is a cactus, and $\vol\mathrm{RMet}(\nabla G)$ with
$O(d^4)$ arithmetic operations, and with $O(d^3)$ if $G$ is a cactus.
\end{corollary}
\begin{proof}
The covariance map $x_{ia}=x_i$\,, $x_{ij}=x_i+x_j-2y_{ij}$ is a linear bijection from
$\R^{V\cup E}$ onto itself whose determinant has absolute value $2^{|E|}$. It maps the vertex of
$\PP(G)$ given by $x\in\{0,1\}^V$ to the cut vector of the cut of $\nabla G$ that separates
$\{i:x_i=1\}$ from the other vertices, so it maps $\PP(G)$ onto $\mathrm{Cut}(\nabla G)$; see
also~\cite[Section~5.2]{DezaLaurent97}. For the rooted metric polytope, a coordinate $x_{ij}$ with
$ij\notin E$ occurs only in the triangle $\{i,j,a\}$, and, because $x_{ia},x_{ja}\in[0,1]$ by the
triangle inequalities of edges at $i$ and at $j$, the value $x_{ij}=x_{ia}+x_{ja}-2x_{ia}x_{ja}$
satisfies its four inequalities. So $\mathrm{RMet}(\nabla G)$ is defined by the triangle
inequalities of the triangles $\{i,j,a\}$ with $ij\in E$, and the covariance map carries these
onto the inequalities~\eqref{eq:Q} that define $\QQ(G)$. The bounds follow from
Theorem~\ref{thm:main} and Proposition~\ref{prop:qcost}.
\end{proof}

\noindent In particular, the volume of the cut polytope of the suspension of a series-parallel
graph, the form in which the question of~\cite[Section~8]{LeeSkipper20} is restated
in~\cite[Section~7]{ADD26}, can be computed in polynomial time. Closed forms for the suspensions
of stars, paths, and cycles are collected in~\cite[Table~2]{ADD26}.

\subsection{A worked example}\label{sec:example}
We carry the running example of Section~\ref{sec:sprunning} through the algorithm:
$G$ is $K_{2,3}$ with a triangle $T$ glued at the cut vertex $a$ (Figure~\ref{fig:running}), with
$n=7$, $m=9$, and $d=16$. Figure~\ref{fig:example} shows the three steps that involve the
machinery of Sections~\ref{sec:chain}--\ref{sec:algo}.

\emph{Block-cut tree} (Figure~\ref{fig:running}(b)). We root the block-cut tree at $v^\ast=s$, so
that $a$ is a child vertex of $B$ and $T$ hangs below $a$. The triangle $T$ is edge-closed; in the
variables $(u_a,u_{c_1},\zeta)$
with $\zeta=z-\tfrac12$\,, its state is
\[
S=\tfrac14-u_a^2-u_{c_1}^2-\zeta^2-4u_au_{c_1}\zeta,
\]
which is symmetric in its three variables. Integrating out $(u_{c_1},z)$ as in
Section~\ref{sec:cutvertex}, we obtain the vertex weight
\[
\omega_a(u)=\Phi_T(u)=\tfrac1{12}\bigl(1-4u^2\bigr)^2;
\]
up to the factor $2^{-3}$ that converts cut coordinates, this is the marginal
$\Phi^\PP_{C_3,a}$ of Proposition~\ref{prop:rhotri}.

\emph{Decomposition of $B$} (Figure~\ref{fig:running}(c)). The vertex $a$ is the internal vertex
of the chain $s,a,t$, so its weight $\omega_a$ enters exactly when we close that chain.

\emph{The theta} (Figure~\ref{fig:example}(a)). Closing either chain of $H$ by a virtual edge $sa$ yields a
triangle, with state $S$ as above in the variables $(u_s,u_a,\zeta)$. From
Lemma~\ref{lem:chaindens}, the density of $1-d_1$ for a path of length two is
$f(z)=(1-z)^2-(u_s+u_a)^2$. Substituting $S_1=S_2=S$ and $f_1=f_2=f$ in
Theorem~\ref{thm:theta}, we obtain
\[
m_0^H(w)=2S(w)\bigl(f(w)-\partial_wS(w)\bigr)=2S(w)\bigl(w^2-\clo^2\bigr),
\]
which is visibly nonnegative for $\clo\le w\le\chii$\,. With $v=|u|$, the mass of $H$ on
$v_s\ge v_a$ is
\[
G_H=\tfrac1{10}(1-2v_s)^2\bigl(1+4v_s+2v_s^2-8v_s^3-10v_a^2-40v_sv_a^2+40v_s^2v_a^2+40v_a^4\bigr).
\]
Because $H$ is a $4$-cycle, integrating $G_H$ and applying~\eqref{eq:jac} recovers
$\vol\PP(C_4)=1/630$ of~\cite[Theorem~19]{LeeSkipper20}.

\emph{Closing the cycle} (Figure~\ref{fig:example}(b)). The cycles through $st$ consist of $st$, a path through $H$,
and the edge $at$. The chain has two segments, so there are four odd labellings. From
Lemma~\ref{lem:chainswitch}, two representatives suffice, each with multiplicity $2$: $A=\{st\}$,
which is the event $d_0<z$ for the chain, and $A=\{\sigma_1\}$, which is the event $d_1<1-z$.

\emph{Rotated coordinates} (Figure~\ref{fig:example}(c)). The cut region of $A$ is the set of staircases
$p_0\le p_1\le p_2$\,. The start $p_0=(u_s,-u_s)$ is fixed, the point $p_1$ that records the
vertex $a$ ranges over the box between $p_0$ and $p_2$\,, and the endpoint $p_2$ slides along the
line $a-b=2e$ with $e=\pm u_t$\,. In the box integral of Proposition~\ref{prop:box}, the density of
$\sigma_1$ is $m_0^H$ or $m_1^H$\,, evaluated at $w_1=S_1$\,, and the vertex weight
$\omega_a(\tilde u_1)$ multiplies the integrand at $p_1$\,. In the $\QQ$-fiber
(Proposition~\ref{prop:qfib}), $\omega_a$ multiplies the first transfer step.

\emph{Result.} From~\eqref{eq:close}, the root state of $B$ is
$g_B=Q_{\rm fib}-2\vol C_{\{st\}}-2\vol C_{\{\sigma_1\}}$\,, a polynomial of degree $12$ with
$197$ terms. Without the weight, the state of $K_{2,3}$ alone has degree $8=11-3$; the weight
$\omega_a$ contributes its degree $4$. Integrating out $(u_t,z)$ gives the block marginal
$\Phi_B(u_s)$, and
\[
\vol\PP(G)=2^{-9}\int_{-1/2}^{1/2}\Phi_B(u)\,du=\frac{653}{908107200}\,.
\]
Every choice of root of the block-cut tree gives the same rational. The recursion of Section~\ref{sec:volQ} gives
$\vol\QQ(G)=64451/48306585600$, so $\vol\PP(G)/\vol\QQ(G)\approx0.5390$.

This ratio is revealing. For $K_{2,3}$ alone, the same computation gives
$\vol\PP(K_{2,3})=1/11880$ and $\vol\PP/\vol\QQ=16/19$, and for the triangle alone the ratio is
$2/3$. If the two blocks were disjoint, the ratios would multiply to $32/57\approx0.5614$. Gluing
them at $a$ lowers the ratio to $0.5390$: both blocks constrain the shared coordinate $x_a$\,, so
their losses do not factor. This interaction is exactly what the vertex weights capture, and
Section~\ref{sec:comp} explains it: the triangle's inequalities cut off volume mostly near
$x_a=\tfrac12$ (Proposition~\ref{prop:rhotri}), where every block glued at $a$ concentrates its
mass (Lemma~\ref{lem:unimodal}), so a hanging triangle cuts off at least a third of $\QQ(G)$
(Corollary~\ref{cor:pendanttri}).

\begin{figure}[ht]
\centering
\tikzset{
  v/.style={circle,fill,inner sep=1.3pt},
  cut/.style={circle,draw,very thick,fill=white,inner sep=1.8pt},
  lab/.style={font=\small},
  root/.style={very thick,red!75!black},
  odd/.style={very thick,red!75!black,dashed},
  seg/.style={line width=2.2pt,blue!55!black},
  tn/.style={draw,rounded corners=2pt,inner sep=2pt,font=\scriptsize,fill=white},
  bn/.style={draw,rounded corners=2pt,inner sep=3pt,font=\small,fill=white}
}
\begin{minipage}[b]{0.40\textwidth}\centering
\begin{tikzpicture}[scale=1.0]
  \coordinate (S) at (0,0); \coordinate (A) at (3,0);
  \coordinate (M1) at (1.5,0.9); \coordinate (M2) at (1.5,-0.9);
  \draw (S) -- (M1) -- (A); \draw (S) -- (M2) -- (A);
  \draw[dashed,thick,gray] (S) -- (A) node[midway,above,lab,gray] {virtual $sa$};
  \foreach \p/\n/\pos in {S/s/left,A/a/right,M1/m_1/above,M2/m_2/below}
    \node[v,label={[lab]\pos:$\n$}] at (\p) {};
  \node[lab] at (0.45,0.75) {$C_1$}; \node[lab] at (0.45,-0.75) {$C_2$};
\end{tikzpicture}\\[2pt]
{\small (a) the theta $H$ and its virtual edge}
\end{minipage}

\vspace{10pt}
\begin{minipage}[b]{0.40\textwidth}\centering
\begin{tikzpicture}[scale=0.85]
  \foreach \dx/\which/\cap in {0/1/{$A=\{st\}$},3.4/2/{$A=\{\sigma_1\}$}} {
    \begin{scope}[xshift=\dx cm]
      \coordinate (S) at (0,0); \coordinate (A) at (2.1,0); \coordinate (T) at (1.05,1.35);
      \ifnum\which=1
        \draw[seg] (S) -- (A); \draw[odd] (T) -- (S);
      \else
        \draw[seg,red!75!black,dashed] (S) -- (A); \draw[very thick] (T) -- (S);
      \fi
      \draw[thick] (A) -- (T);
      \node[lab,below] at (1.05,-0.05) {$\sigma_1=H$};
      \node[lab,right] at (1.6,0.75) {$\sigma_2$};
      \node[lab,left] at (0.5,0.75) {$st$};
      \foreach \p/\n/\pos in {S/s/left,T/t/above}
        \node[v,label={[lab]\pos:$\n$}] at (\p) {};
      \node[cut,label={[lab]right:$a$}] at (A) {};
      \node[lab] at (1.05,-0.8) {\cap};
    \end{scope}}
\end{tikzpicture}\\[2pt]
{\small (b) the two representative odd labellings (dashed: odd)}
\end{minipage}\hfill
\begin{minipage}[b]{0.57\textwidth}\centering
\begin{tikzpicture}[scale=3.4,>={Stealth[length=2mm]}]
  \draw[->] (-0.15,0) -- (0.95,0) node[right,lab] {$a$};
  \draw[->] (0,-0.25) -- (0,0.75) node[above,lab] {$b$};
  \coordinate (P0) at (0.10,-0.10); \coordinate (P1) at (0.36,0.21); \coordinate (P2) at (0.66,0.56);
  \fill[blue!10] (P0) rectangle (P2);
  \draw[blue!40] (P0) rectangle (P2);
  \draw[dashed,gray] (0.10,0.00) -- (0.80,0.70) node[above,lab,gray] {$a-b=2e$};
  \node[circle,draw,gray,inner sep=1.1pt,fill=white] at (0.10,0.00) {};
  \draw[very thick,->,red!60!black] (0.46,0.36) -- (0.72,0.62);
  \draw[thick,->] (P0) -- (P0 -| P1) -- (P1);
  \draw[thick,->] (P1) -- (P1 -| P2) -- (P2);
  \node[v,label={[lab]below right:$p_0$}] at (P0) {};
  \node[cut,label={[lab]above left:$p_1$}] at (P1) {};
  \node[lab,green!45!black,anchor=west] at (0.40,0.10) {$\omega_a(\tilde u_1)$};
  \node[v,label={[lab]left:$p_2$}] at (P2) {};
\end{tikzpicture}\\[2pt]
{\small (c) rotated coordinates: $p_1$ (vertex $a$) ranges over\\ the shaded box, carrying the weight $\omega_a$; $p_2$ slides along the line}
\end{minipage}
\caption{The algorithm on the running example of Figure~\ref{fig:running}
(Section~\ref{sec:example}).}
\label{fig:example}
\end{figure}

\subsection{Cut polytopes of series-parallel graphs}\label{sec:cutG}
The same recursion, with every vertex coordinate fixed at $x_v=\tfrac12$ (that is, $u_v=0$)
instead of integrated out, computes the volume of the cut polytope of $G$ itself. For forests,
cycles, cacti, and necklaces (cycles with a cycle glued at each vertex), closed forms for this
volume are given in~\cite[Section~3]{ADD26}; we treat every series-parallel graph. For $S\subseteq V$, the \emph{cut vector} of $S$ is the vector in
$\{0,1\}^E$ whose entry for $e$ is $1$ exactly when $e$ has one end in $S$, and the \emph{cut
polytope} $\mathrm{Cut}(G)\subseteq\R^E$ is the convex hull of the cut vectors. If $G$ has no
$K_5$ minor, in particular if $G$ is series-parallel, then, by the theorem
of~\cite{BarahonaMahjoub86}, $\mathrm{Cut}(G)$ is the set of $z\in[0,1]^E$ that satisfy
\[
\sum_{e\in F}(1-z_e)+\sum_{e\in E(C)\setminus F}z_e\ge1
\]
for every cycle $C$ of $G$ and every odd set $F\subseteq E(C)$. When every vertex coordinate
$u_v$ is $0$, each interval of Lemma~\ref{lem:edge} is $[0,1]$, and, by Lemma~\ref{lem:oc}, these
are the odd-cycle inequalities. So, by Theorem~\ref{thm:padberg}, $\mathrm{Cut}(G)$ is the fiber
of $\PP(G)$ in cut coordinates over $u=0$, and $\vol\mathrm{Cut}(G)$ is the volume of that fiber.

For a two-terminal piece $H$ of the decomposition, fix all vertex coordinates of $H$ at $0$, so
that a configuration of $H$ is a point $z\in[0,1]^{E(H)}$. Nothing now couples the segments of a
chain, so their parity distances are independent, and the distribution of a sum of segment lengths
is the convolution of their distributions; this is why the recursion below is built from
convolutions. We write $G^0_H$ for the volume of
the feasible configurations and $m^0_H(w)$ for the density of $d_0(H)$ at $w\in[0,1]$, and we use
the convolution $(f\star g)(w)=\int_0^wf(t)\,g(w-t)\,dt$ of polynomials, which is again a
polynomial.

\begin{theorem}\label{thm:cutG}
Let $G$ be a series-parallel graph with $m$ edges. For the nodes of the decomposition of each
$2$-connected block of $G$, the numbers $G^0_H$ and the polynomials $m^0_H$ satisfy the following.
\begin{enumerate}[label=\textup{(\alph*)}]
\item For an edge, $G^0=1$ and $m^0=1$.
\item For a chain $C$ with segments $\sigma_1,\dots,\sigma_k$\,, put $G^0_C=\prod_jG^0_{\sigma_j}$\,,
$h_C=m^0_{\sigma_1}\star\cdots\star m^0_{\sigma_k}$\,, $D_C(q)=2^{k-1}\int_0^qh_C$\,, and
$S_C(z)=G^0_C-D_C(z)-D_C(1-z)$.
\item For an edge-closed P-node $H$ with chains $C_1,\dots,C_r$\,, $m^0_H=\prod_jS_{C_j}$ and
$G^0_H=\int_0^1m^0_H$\,.
\item For a theta $H$ with chains $C_1,\dots,C_r$\,,
\[
m^0_H=\sum_iD_{C_i}'\prod_{j\ne i}S_{C_j}\,,\qquad
G^0_H=\prod_jS_{C_j}(0)+\sum_i\int_0^1D_{C_i}'(1-z)\prod_{j\ne i}S_{C_j}(z)\,dz.
\]
\end{enumerate}
Then $\vol\mathrm{Cut}(G)=\prod_BG^0_B$\,, where $B$ ranges over the $2$-connected blocks of $G$
and $G^0_B$ is the value at the root of its decomposition. In particular, $\vol\mathrm{Cut}(G)$ can
be computed with $O(m^3)$ arithmetic operations on rationals of bit length $O(m^2)$.
\end{theorem}

\begin{proof}
All vertex coordinates are fixed at $0$ throughout. The switching $\sigma_v$ of
Lemma~\ref{lem:switch} fixes $u=0$ for every vertex $v$, preserves volume, and maps feasible
configurations to feasible ones. At a terminal of $H$ it exchanges $d_0(H)$ and $d_1(H)$, so the two
have the same distribution; at an internal vertex of a chain it changes the parities of the two
segments at that vertex. By Lemma~\ref{lem:parity}(b), $d_0+d_1\ge1$ for every feasible piece.

\emph{Chains.} Because the vertex coordinates are fixed, the segments of a chain $C$ are
independent. For $q\le1$ and a set $A$ of segments, let $E_A(q)$ be the event
$\sum_{j\in A}d_1(\sigma_j)+\sum_{j\notin A}d_0(\sigma_j)<q$. Then $d_0(C)<q$ exactly when $E_A(q)$
holds for some even $A$, and $d_1(C)<q$ exactly when it holds for some odd $A$. Two such events for
different sets $A$ and $A'$ of the same parity cannot hold together, because every segment in
$A\triangle A'$ contributes $d_0+d_1\ge1$ to the sum of the two left-hand sides, so that sum is at
least $|A\triangle A'|\ge2\ge2q$, whereas it is less than $2q$. Switching at internal vertices maps $E_A(q)$ onto $E_\emptyset(q)$
for every even $A$ and onto $E_{\{\sigma_1\}}(q)$ for every odd $A$, and, because $d_1(\sigma_1)$ and
$d_0(\sigma_1)$ have the same distribution, the last two events have the same volume. By
independence, the volume of $E_\emptyset(q)$ is $\int_{w_1+\dots+w_k<q}\prod_jm^0_{\sigma_j}(w_j)\,dw$,
which is $\int_0^qh_C$ because $q\le1$. Hence $d_0(C)<q$ and $d_1(C)<q$ each have volume $D_C(q)$.
The same argument, applied to sets of different parity, shows that $d_0(C)+d_1(C)\ge1$, so the
events $d_0(C)<z$ and $d_1(C)<1-z$ are disjoint, and $S_C(z)$ is the volume of the event
$d_0(C)\ge z$ and $d_1(C)\ge1-z$.

\emph{Edge-closed P-nodes.} A configuration is feasible exactly when every chain is feasible and
$d_0(C_j)\ge z_{st}$ and $d_1(C_j)\ge1-z_{st}$ for every $j$: these are the cycles through $st$,
and they imply $d_0(C_i)+d_1(C_j)\ge1$ for the cycles through two chains. On feasible
configurations $d_0(H)=z_{st}$ by Lemma~\ref{lem:parity}(c), which gives (c).

\emph{Thetas.} Put $\ell_j=1-d_1(C_j)$. A configuration is feasible exactly when every chain is
feasible and $d_0(C_i)\ge\ell_j$ for all $i\ne j$. Off a set of measure zero, $d_0(H)=\min_id_0(C_i)$
is attained by a unique $i$; if its value is $w$, feasibility means $d_0(C_j)\ge w$ and $\ell_j\le w$
for all $j\ne i$, because $\ell_i\le d_0(C_i)\le d_0(C_j)$. This gives $m^0_H$\,. For the mass, let
$M=\max_j\ell_j\le1$. If $M\le0$, every configuration of feasible chains is feasible, which gives
$\prod_jS_{C_j}(0)$. Otherwise, off a set of measure zero, $M=\ell_i=z\in(0,1]$ for a unique $i$,
whose density is $D_{C_i}'(1-z)$, and feasibility means $d_0(C_j)\ge z$ and $\ell_j\le z$ for
$j\ne i$, which gives the integral.

\emph{Blocks.} The root of the decomposition of a $2$-connected block is edge-closed, and
integrating $m^0$ over $z_{st}$ gives the volume of the fiber of $\PP(B)$ over $u=0$, that is,
$\vol\mathrm{Cut}(B)$. Blocks share no edge coordinates, so the fiber of $\PP(G)$ over $u=0$ is the
product of the fibers of its blocks, and a bridge contributes $[0,1]$.

\emph{Complexity.} By induction, $m^0_H$ has degree less than $|E(H)|$, and $h_C$ and $S_C$ have
degree at most $|E(C)|$. A convolution or product of polynomials of degrees $a$ and $b$ costs
$O(ab)$, and a substitution $z\mapsto1-z$ costs $O(a^2)$. At a chain, the convolutions and
substitutions cost $O(|E(C)|^2)$. At a P-node $H$, the products $\prod_{j\ne i}S_{C_j}$ for all $i$
are obtained from prefix and suffix products with $O(r)$ products, each with one factor of degree
at most $|E(C_j)|$, so they cost $O(|E(H)|^2)$. The edge sets of the pieces of the decomposition form a laminar
family (any two are disjoint or nested) of depth at most $m$, so $\sum_H|E(H)|\le m^2$ and
$\sum_H|E(H)|^2\le m^3$. For the bit lengths, as in the proof of Proposition~\ref{prop:bits}, a
convolution, a substitution, or an integration adds $O(m)$ to the bit length, because the
coefficients $a!\,b!/(a+b+1)!$ and the binomial coefficients involved have $O(m)$ bits, and a
product adds the bit lengths of its factors; by induction, the coefficients at $H$ have bit length
$O(m\,|E(H)|)$.
\end{proof}

For example, for $r\ge2$, $K_{2,r}$ is a theta with $r$ chains of two edges. For such a chain, $h=w$,
$D(q)=q^2$, and $S(z)=2z(1-z)$, so, with $S(0)=0$,
\[
G^0=r\int_0^12(1-z)\bigl(2z(1-z)\bigr)^{r-1}dz=r\,2^r\int_0^1z^{r-1}(1-z)^r\,dz.
\]
Evaluating the Beta integral gives the following.

\begin{corollary}\label{cor:K2r}
For every $r\ge1$, $\vol\mathrm{Cut}(K_{2,r})=2^r(r!)^2/(2r)!$.
\end{corollary}

\noindent For $r=1$, $K_{2,1}$ is a path, and both sides equal $1$.

\begin{remark}\label{rem:suspforest}
If $F$ is a forest, then $\nabla F$ is series-parallel, because a $K_4$ minor of $\nabla F$ would
give a $K_3$ minor of $F$. So Theorem~\ref{thm:cutG} computes
$\vol\mathrm{Cut}(\nabla F)=2^{|E(F)|}\vol\PP(F)$ (Corollary~\ref{cor:rmet}) by a recursion that is
independent of the one behind Theorem~\ref{thm:main}. For a star with $n$ edges, $\nabla F$ is the book $K_{1,1,n}$\,, that
is, $n$ triangles sharing one edge; for a path, $\nabla F$ is a fan, a path together with a vertex
adjacent to all of its vertices. For these, it reproduces the closed forms
of~\cite[Theorems~7 and~10]{LeeSkipper20}.
\end{remark}

% ---------------------------------------------------------------------------

\section{How the volume ratio behaves}\label{sec:comp}

The relaxation $\QQ$ is the one that is often used in the practice of mixed-integer optimization,
so it is worth knowing how much of it $\PP$ occupies. We write
\[
\mathrm{ratio}(G)=\frac{\vol\PP(G)}{\vol\QQ(G)}\,,
\]
the fraction of $\QQ(G)$ that survives the odd-cycle inequalities. For a single cycle,
\cite[Theorems~13 and~19]{LeeSkipper20} give
\[
r(\ell):=\mathrm{ratio}(C_\ell)=1-\frac{4^{\ell-1}}{\ell\,A_{2\ell-1}}\,,
\]
where $A_k$ is the number of alternating permutations of $k$ elements; see~\cite{Stanley10}. The first values are $A_1,\dots,A_7=1,1,2,5,16,61,272$, so that, for
example, $r(3)=1-16/(3A_5)=2/3$, and $r(\ell)\to1$ rapidly.
For a disjoint union of cycles, the ratios multiply. The authors of~\cite[Section~7]{LeeSkipper20}
concluded from this, and from Monte Carlo experiments on one family, that odd-cycle inequalities
matter far more for short cycles than for long ones. We prove this in a strong form: in every
graph, the odd-cycle inequalities of a cycle of length $\ell$ cut off at most a fraction
$2^{\ell-1}/\ell!$ of $\QQ(G)$, and this bound is sharp (Section~\ref{sec:cap}). The rest of the
section works with the fibers of $\PP(G)$ and $\QQ(G)$ over a single vertex
(Section~\ref{sec:marginals}). With them, we show that a triangle hanging from a vertex always
cuts off at least as much as any $4$-cycle can, and that a hanging cycle of any length matters
about as much as it does on its own (Section~\ref{sec:hangtri}). We also show that the ratios
do not multiply once cycles share a vertex: a long cycle lowers the ratio of its neighbor by a
definite factor (Section~\ref{sec:long}), and in a large flower each cycle lowers the ratio
strictly more than on its own (Section~\ref{sec:flowers}). Finally, in the worst case, the odd-cycle
inequalities of short cycles alone leave $\PP(G)$ an exponentially small fraction of the
relaxation that they define, because the small fractions that many long cycles cut off compound
(Section~\ref{sec:short}).

\subsection{Short cycles cut more}\label{sec:cap}
We say that a set of inequalities \emph{cuts off} the fraction of $\QQ(G)$ that violates at least
one of them. Our main result on the volume ratio holds for every graph, not only for
series-parallel ones, and it makes precise the heuristic of Lee and Skipper that odd-cycle
inequalities matter far more for short cycles than for long ones.

\begin{theorem}\label{thm:cap}
Let $G$ be a graph, and let $C$ be a cycle of $G$ of length $\ell$. Then each odd-cycle inequality
of $C$ cuts off at most a fraction $1/\ell!$ of $\QQ(G)$, and all of them together cut off at most
a fraction $2^{\ell-1}/\ell!$. In particular, if $G$ is series-parallel, then
\[
\mathrm{ratio}(G)\ge1-\sum_C\frac{2^{|C|-1}}{|C|!}\,,
\]
where $C$ ranges over the cycles of $G$. The bound is sharp: if $G_k$ is obtained from $C_\ell$ by
attaching $k$ pendant edges at every vertex, then
$\mathrm{ratio}(G_k)\to1-2^{\ell-1}/\ell!$ as $k\to\infty$.
\end{theorem}
\begin{proof}
By Lemma~\ref{lem:edge}, $\QQ(G)$ consists of the points $(u,z)$ with $u\in[-\tfrac12,\tfrac12]^V$
and $z_e$ in an interval $I_e(u)$ of length $L_e(u)=1-2\max(|u_i|,|u_j|)$ for every edge $e=ij$,
independently over the edges. Fix $u$ and an odd set $A\subseteq E(C)$, and write the cycle as
$v_0,v_1,\dots,v_\ell=v_0$ with $e_j=v_{j-1}v_j$\,. Put $w_j=1-z_{e_j}$ for $e_j\in A$ and
$w_j=z_{e_j}$ otherwise, and let $\tilde u_j=\pm u_{v_j}$\,, with $\tilde u_0=u_{v_0}$ and with
the sign changing exactly across the edges of $A$; because $|A|$ is odd, $\tilde u_\ell=-u_{v_0}$\,. By Lemmas~\ref{lem:edge} and~\ref{lem:oc}, $w_j$ ranges over an interval
$[a_j,a_j+L_{e_j}]$ with $a_j=|\tilde u_j-\tilde u_{j-1}|$, and $\OC(A)$ fails exactly when
$\sum_jw_j<1$. Let $R_A=1-\sum_ja_j$\,. Going around $C$ from $v_j$ back to $v_{j-1}$\,, the values
of $\tilde u$ run from $\tilde u_j$ to $-\tilde u_{j-1}$\,, so $\sum_{i\ne j}a_i\ge|\tilde u_j+\tilde
u_{j-1}|$, and therefore
\[
R_A\le1-|\tilde u_j-\tilde u_{j-1}|-|\tilde u_j+\tilde u_{j-1}|=L_{e_j}
\qquad\text{for every }j.
\]
Hence the upper ends of the intervals never bind, and the set of $z|_C$ in the box
$\prod_{e\in C}I_e(u)$ that violates $\OC(A)$ is a simplex of volume
\[
\frac{(R_A)_+^\ell}{\ell!}\le\frac1{\ell!}\prod_{e\in C}L_e(u).
\]
We multiply by $\prod_{e\notin C}L_e(u)$ and integrate over $u$. This shows that $\OC(A)$ cuts
off at most $\vol\QQ(G)/\ell!$ in cut coordinates, hence at most the fraction $1/\ell!$. There are
$2^{\ell-1}$ odd sets $A$. If $G$ is series-parallel, then $\PP(G)$ is cut out of $\QQ(G)$ by
the odd-cycle inequalities of all cycles (Theorem~\ref{thm:padberg}), and the displayed bound
follows by adding the fractions.

For sharpness, the sets on which the various $\OC(A)$ fail are disjoint: for odd $A\ne A'$, the
sum of the two left-hand sides of Lemma~\ref{lem:oc} is at least $|A\triangle A'|\ge2$. So, conditionally on $u$, the fraction of the box that is cut
off is exactly $\Lambda(u)=\sum_A(R_A)_+^\ell\big/\bigl(\ell!\prod_{e\in C}L_e(u)\bigr)$, which is
continuous near $u=0$, with $\Lambda(0)=2^{\ell-1}/\ell!$. A pendant edge at $w$ contributes the
factor $\tfrac12-2u_w^2$, by~\eqref{eq:bridge}, so in $\QQ(G_k)$ the density of $u|_C$ is
proportional to $\prod_{e\in C}L_e(u)\prod_{w\in V(C)}(\tfrac12-2u_w^2)^k$. Outside the cube
$\{\max_w|u_w|<\delta\}$, the second product is at most $(1-4\delta^2)^k$ times its value at
$u=0$, and inside the cube $\{\max_w|u_w|<\delta/k\}$ it is at least $(1-4\delta^2/k^2)^{k\ell}$
times that value; so the density concentrates at $u=0$ as $k\to\infty$, and the fraction of $\QQ(G_k)$ cut off
by the odd-cycle inequalities of $C_\ell$ tends to $\Lambda(0)$. Because $C_\ell$ is the only cycle
of $G_k$\,, from Theorem~\ref{thm:padberg}, this fraction is $1-\mathrm{ratio}(G_k)$.
\end{proof}

Lengthening a cycle from $\ell$ to $\ell+1$ edges multiplies the cap $2^{\ell-1}/\ell!$ by
$2/(\ell+1)$. On its own, a cycle cuts off only $1-r(\ell)$; inside a graph, the fraction that it cuts off
can be greater, but, by Theorem~\ref{thm:cap}, never more than the cap.

The proof of sharpness shows more. Let $G_k$ be obtained from a series-parallel graph $G$ by
attaching $k$ pendant edges at every vertex. Each pendant edge at $v$ multiplies the fibers over
the vertex coordinates $u$ of $G$ by $\tfrac12-2u_v^2$, which is largest at $u_v=0$, that is, at
$x_v=\tfrac12$; so, as $k\to\infty$, the distribution of the vertex coordinates concentrates at $u=0$.
There, the fiber of $\QQ(G)$ is the unit cube and that of $\PP(G)$ is $\mathrm{Cut}(G)$
(Section~\ref{sec:cutG}), and, because the fiber volumes are continuous at $u=0$,
$\mathrm{ratio}(G_k)\to\vol\mathrm{Cut}(G)$. For $G=C_\ell$\,, this is the sharpness statement,
with $\vol\mathrm{Cut}(C_\ell)=1-2^{\ell-1}/\ell!$; see also~\cite[Proposition~3.4]{ADD26}. At $u=0$, Theorem~\ref{thm:cap} also gives
$\vol\mathrm{Cut}(G)\ge1-\sum_C2^{|C|-1}/|C|!$\,.

\subsection{Fiber marginals}\label{sec:marginals}
For a graph $G$ with a vertex $v$ and for $u\in[-\tfrac12,\tfrac12]$, let $\Phi^\PP_{G,v}(u)$ and
$\Phi^\QQ_{G,v}(u)$ be the $(d-1)$-dimensional volumes of the fibers of $\PP(G)$ and $\QQ(G)$ over
$u_v=u$, and let $\rho_{G,v}=\Phi^\PP_{G,v}/\Phi^\QQ_{G,v}$ on $(-\tfrac12,\tfrac12)$. For a series-parallel graph, they are,
up to constant factors, the products of the marginals of Section~\ref{sec:cutvertex} over the
blocks at $v$; in particular, they are polynomials in $|u|$.
Because $\mathrm{ratio}(G)=\int\Phi^\PP_{G,v}\big/\!\int\Phi^\QQ_{G,v}$\,, the ratio is the average of
$\rho_{G,v}$ against the probability density proportional to $\Phi^\QQ_{G,v}$\,.

If $G_1$ and $G_2$ are series-parallel graphs that share exactly the vertex $v$ and no edge, then
every cycle of $G=G_1\cup G_2$ lies in $G_1$ or in $G_2$\,, so by Theorem~\ref{thm:padberg} the fiber
of $\PP(G)$ over $u_v=u$ is the product of the fibers of $\PP(G_1)$ and $\PP(G_2)$, and the same
holds for $\QQ$. Hence
\begin{equation}\label{eq:fiberprod}
\Phi^\PP_{G,v}=\Phi^\PP_{G_1,v}\,\Phi^\PP_{G_2,v}\,,\qquad
\Phi^\QQ_{G,v}=\Phi^\QQ_{G_1,v}\,\Phi^\QQ_{G_2,v}\,,\qquad
\rho_{G,v}=\rho_{G_1,v}\,\rho_{G_2,v}\,.
\end{equation}

We use Chebyshev's integral inequality in the following form. Let $\mu$ be a probability measure
on $[0,\tfrac12]$, and let $\varphi$ be nondecreasing and $\psi$ nonincreasing. Then
$\int\varphi\psi\,d\mu\le\int\varphi\,d\mu\int\psi\,d\mu$, and equality forces $\varphi$ or $\psi$
to be constant $\mu$-almost everywhere. This is classical (see, for
example,~\cite[Section~2.17]{HLP52}); it follows by integrating the pointwise inequality
$(\varphi(x)-\varphi(y))(\psi(x)-\psi(y))\le0$ against $\mu\times\mu$. In probabilistic terms, it
says that a nondecreasing and a nonincreasing function of the same random variable are negatively
correlated. We apply it to even functions of $u$, regarded as functions of $|u|$.

\begin{lemma}\label{lem:unimodal}
For every graph $G$ and every vertex $v$, the functions $\Phi^\PP_{G,v}$ and $\Phi^\QQ_{G,v}$ are
even and nonincreasing in $|u|$. If $v$ is not isolated, they vanish at $u=\pm\tfrac12$ and are
positive inside. If $G$ has a cycle, then $\rho_{G,v}(0)<1$.
\end{lemma}
\begin{proof}
We use Brunn's concavity theorem, which is a consequence of the Brunn--Minkowski inequality;
see~\cite{Gardner02}. The theorem states that the $(d-1)$st root of the volume of the sections of a convex body in $\R^d$ by a
family of parallel hyperplanes is a concave function of the position of the hyperplane; we section by the hyperplanes $\{u_v=u\}$. The switching $\sigma_v$ of
Lemma~\ref{lem:switch} maps the fiber over $u$ onto the fiber over $-u$ and preserves volume, so
both functions are even. An even function on $[-\tfrac12,\tfrac12]$ that has a concave positive
power is nonincreasing in $|u|$. At $u=\pm\tfrac12$, each edge $vw$ forces $y_{vw}=0$ or
$y_{vw}=x_w$\,, so the fiber has lower dimension; inside, the fibers have positive volume, because
$\PP(G)$ and $\QQ(G)$ are full-dimensional. For the last claim, let $C$ be a cycle and $f$ an edge
of $C$. By Lemma~\ref{lem:edge}, the points with all $u_w$ close to $0$, with $z_f$ close to $1$,
with $z_e$ close to $0$ for $e\in E(C)\setminus f$, and with $z_e$ close to $\tfrac12$ for all
other edges lie in $\QQ(G)$; by Lemma~\ref{lem:oc}, they violate $\OC(\{f\})$. They form an open
set that meets the fiber over $u_v=0$ in a set of positive volume, so
$\Phi^\PP_{G,v}(0)<\Phi^\QQ_{G,v}(0)$.
\end{proof}

\subsection{Hanging cycles}\label{sec:hangtri}
For a cycle that hangs from a single vertex, Theorem~\ref{thm:cap} has a counterpart from
below. We begin with the triangle, whose marginals are explicit.

\begin{proposition}\label{prop:rhotri}
For every vertex $v$ of $C_3$ and with $t=|u|$,
\[
\Phi^\PP_{C_3,v}(u)=\frac{(1-4t^2)^2}{96}\,,\qquad
\Phi^\QQ_{C_3,v}(u)=\frac{(1-2t)^2(1+4t+2t^2-8t^3)}{60}\,,
\]
\[
\rho_{C_3,v}(u)=\frac{5(1+2t)^2}{8\,(1+4t+2t^2-8t^3)}\,,
\]
and $\rho_{C_3,v}$ is strictly increasing in $t$, from $5/8$ at $x_v=\tfrac12$ to $1$ at
$x_v\in\{0,1\}$.
\end{proposition}
\begin{proof}
By evenness we may take $u=t\ge0$. We use the kernel $\chi$ of the proof of
Lemma~\ref{lem:cycmarg} below; by Lemma~\ref{lem:edge}, $\chi(u,a)=\tfrac12\bigl(1-2\max(|u|,|a|)\bigr)$.

\emph{The $\QQ$-marginal.} Integrating out the $y_e$ and the two other vertices gives
$\Phi^\QQ_{C_3,v}(u)=\iint\chi(u,a)\,\chi(a,b)\,\chi(b,u)\,da\,db$. The integrand is even in $a$
and in $b$, and with $T=1-2t$, $P=1-2|a|$, and $Q=1-2|b|$, each factor $2\chi$ is the minimum of
the two corresponding variables. Hence
\[
\Phi^\QQ_{C_3,v}(u)=\frac18\int_0^1\!\!\int_0^1\min(T,P)\min(P,Q)\min(Q,T)\,dP\,dQ
=\frac{T^2(5-5T^2+2T^3)}{120}\,,
\]
where the integral is evaluated by splitting the square according to the order of $P$, $Q$, and
$T$. Substituting $T=1-2t$ gives the stated formula.

\emph{The $\PP$-marginal.} By Theorem~\ref{thm:padberg}, $\Phi^\QQ_{C_3,v}-\Phi^\PP_{C_3,v}$ is the
fiber over $u_v=t$ of the four disjoint sets on which an odd-cycle inequality fails. In the
notation of the proof of Theorem~\ref{thm:cap}, with $v_0=v$, such a set meets the box over a
point $(u_{v_1},u_{v_2})$ in a simplex of volume $(R_A)_+^3/3!$ in the coordinates $z$, where,
with $\xi=\tilde u_1$ and $\eta=\tilde u_2$\,,
\[
1-R_A=|\xi-t|+|\eta-\xi|+|\eta+t|=2t+2\bigl((\xi-t)_++(\eta-\xi)_++(-t-\eta)_+\bigr);
\]
the identity holds because $|w|=-w+2w_+$ for each of the three differences. The change of
variables from $(u_{v_1},u_{v_2})$ to $(\xi,\eta)$ only changes signs, and $R_A>0$ forces
$|\xi|,|\eta|<\tfrac12$, so all four sets contribute the same integral over $\R^2$. Converting the $z_e$
into the $y_e$ multiplies by $2^{-3}$, so
\[
\Phi^\QQ_{C_3,v}-\Phi^\PP_{C_3,v}=\frac{4}{8\cdot3!}\,J,\qquad
J=\iint\bigl(R-2\,\mathrm{ex}(\xi,\eta)\bigr)_+^3\,d\xi\,d\eta,
\]
with $R=1-2t$ and the excess $\mathrm{ex}=(\xi-t)_++(\eta-\xi)_++(-t-\eta)_+$\,. At most two of the three positive parts
are nonzero, because $\xi>t$ and $\eta>\xi$ imply $\eta>-t$. We split the plane accordingly:
\begin{itemize}[nosep]
\item where all three vanish, that is, on the triangle $t\ge\xi\ge\eta\ge-t$ of area $2t^2$, the
integrand is $R^3$;
\item where exactly one is positive, with value $a$, the other variable ranges over an interval of
length $2t+a$, and each of the three cases contributes
$\int_0^\infty(2t+a)(R-2a)_+^3\,da=\tfrac14tR^4+\tfrac1{80}R^5$;
\item where exactly two are positive, with values $a$ and $b$, these are coordinates with
Jacobian $1$, and each of the three cases contributes
$\iint_{a,b>0}(R-2a-2b)_+^3\,da\,db=\tfrac1{80}R^5$.
\end{itemize}
Hence $J=2t^2R^3+\tfrac34tR^4+\tfrac3{40}R^5$, and
\[
\Phi^\PP_{C_3,v}(u)=\frac{(1-2t)^2(1+4t+2t^2-8t^3)}{60}-\frac{(1-2t)^3(3+18t+32t^2)}{480}
=\frac{(1-4t^2)^2}{96}\,.
\]
At $t=0$, the difference is $1/160$, in agreement with Proposition~\ref{prop:rhocenter}.

\emph{Monotonicity.} The derivative of $(1+2t)^2/(1+4t+2t^2-8t^3)$ has numerator
$4t(1+2t)(1+6t+4t^2)$, which is positive for $t>0$.
\end{proof}

The odd-cycle inequalities of a triangle thus cut off volume almost entirely near
$x_v=\tfrac12$\,, where the fractional vertex of~\cite[Theorem~15]{LeeSkipper20} lies, and by
Lemma~\ref{lem:unimodal} every graph glued at $v$ reweights the triangle toward exactly that
region. So a hanging triangle cuts off at least as much as it does on its own, which is as much
as any $4$-cycle can cut off anywhere.

\begin{corollary}\label{cor:pendanttri}
Let $G$ be obtained by gluing a triangle $T$ at one of its vertices $v$ to an arbitrary graph $B$.
Then the odd-cycle inequalities of $T$ cut off at least a fraction $1/3$ of $\QQ(G)$, with strict
inequality if $B$ has an edge at $v$; if $B$ is a forest, this says that
$\mathrm{ratio}(G)\le2/3$. Consequently, $T$ cuts off at least as large a fraction of $\QQ(G)$ as
any $4$-cycle cuts off from $\QQ(G')$, for any graph $G'$, and at least $5/2$ times as large a
fraction as any longer cycle.
\end{corollary}
\begin{proof}
The fiber of $\QQ(G)\cap\{\text{odd-cycle inequalities of }T\}$ over $u_v=u$ is the product of the
fiber of $\PP(T)$ and the fiber of $\QQ(B)$. Hence the fraction of $\QQ(G)$ that satisfies them is the average of
$\rho_{T,v}$ against the density proportional to $\Phi^\QQ_{T,v}\Phi^\QQ_{B,v}$\,. By
Proposition~\ref{prop:rhotri}, $\rho_{T,v}$ is strictly increasing in $|u|$, and by
Lemma~\ref{lem:unimodal}, $\Phi^\QQ_{B,v}$ is nonincreasing, and not constant if $B$ has an edge
at $v$. Chebyshev's inequality bounds the average by the average against $\Phi^\QQ_{T,v}$\,, which
is $r(3)=2/3$. The comparisons follow from Theorem~\ref{thm:cap}.
\end{proof}

For longer cycles, what the odd-cycle inequalities of a hanging cycle do is governed, up to
explicit constant factors, by what they do on the cycle alone.

\begin{theorem}\label{thm:hangcyc}
Let $B$ be a series-parallel graph with a vertex $v$, let $d_B=|V(B)|+|E(B)|$, let $\ell\ge3$, and
let $G$ be obtained by gluing $C_\ell$ to $B$ at $v$. Put
$h^\PP_B=\Phi^\PP_{B,v}(0)/\vol\PP(B)$ and $h^\QQ_B=\Phi^\QQ_{B,v}(0)/\vol\QQ(B)$; then
$h^\PP_B,h^\QQ_B\le d_B$\,. The following hold.
\begin{enumerate}[label=\textup{(\alph*)},nosep]
\item The odd-cycle inequalities of $C_\ell$ cut off a fraction of $\QQ(G)$ between
$\tfrac12\bigl(1-r(\ell)\bigr)$ and $h^\QQ_B\bigl(1-r(\ell)\bigr)$.
\item Omitting them, while keeping all other odd-cycle inequalities, enlarges $\PP(G)$ by a
fraction between $\tfrac12\bigl(1-r(\ell)\bigr)$ and $h^\PP_B\bigl(1-r(\ell)\bigr)/r(\ell)$.
\end{enumerate}
Here $1-r(\ell)=4^{\ell-1}/(\ell A_{2\ell-1})\le\pi^{2\ell}/(4\,(2\ell)!)$.
\end{theorem}
\begin{proof}
Write $\Phi^\bullet_\ell=\Phi^\bullet_{C_\ell,v}$ and $\Phi^\bullet_B=\Phi^\bullet_{B,v}$ for
$\bullet\in\{\PP,\QQ\}$, let $V_\ell=\vol\QQ(C_\ell)$, and let $D=\Phi^\QQ_\ell-\Phi^\PP_\ell$\,. All
integrals are over $[-\tfrac12,\tfrac12]$, and $\int D=(1-r(\ell))V_\ell$\,.

\emph{The function $D$ is nonincreasing in $|u|$.} By Theorem~\ref{thm:padberg}, $D$ is the fiber
function over $u_v$ of the union of the $2^{\ell-1}$ sets on which an odd-cycle inequality of
$C_\ell$ fails, and these sets are disjoint (proof of Theorem~\ref{thm:cap}). By
Lemma~\ref{lem:oc}, each of them is $\QQ(C_\ell)$ intersected with an open halfspace, so it is
convex. Each is also invariant under the map that negates every vertex coordinate $u_w$ and fixes
$z$, because this map preserves the constraints of Lemma~\ref{lem:edge}. As in the proof of
Lemma~\ref{lem:unimodal}, Brunn's theorem shows that the fiber function of each set is even and
nonincreasing in $|u|$, and hence so is $D$.

(a) By the fiber product~\eqref{eq:fiberprod}, the fraction of $\QQ(G)$ cut off is
$\int D\,\Phi^\QQ_B\big/\!\int\Phi^\QQ_\ell\Phi^\QQ_B$\,. The functions $D$ and $\Phi^\QQ_B$ are
both nonincreasing in $|u|$ (Lemma~\ref{lem:unimodal}), so Chebyshev's inequality, in the form for
two functions monotone in the same direction, gives
$\int D\,\Phi^\QQ_B\ge\int D\int\Phi^\QQ_B=(1-r(\ell))V_\ell\vol\QQ(B)$; also,
$\int D\,\Phi^\QQ_B\le\Phi^\QQ_B(0)\,(1-r(\ell))V_\ell$\,. For the denominator,
$\Phi^\QQ_\ell\le\Phi^\QQ_\ell(0)=2V_\ell$ by Lemma~\ref{lem:unimodal} and Lemma~\ref{lem:cycmarg}
below, so $\int\Phi^\QQ_\ell\Phi^\QQ_B\le2V_\ell\vol\QQ(B)$, and Chebyshev's inequality gives
$\int\Phi^\QQ_\ell\Phi^\QQ_B\ge V_\ell\vol\QQ(B)$. Combining these bounds gives (a).

(b) By~\eqref{eq:fiberprod}, the fibers of $\PP(G)$ are the products of those of $\PP(C_\ell)$ and
$\PP(B)$, and omitting the inequalities of $C_\ell$ replaces $\PP(C_\ell)$ by $\QQ(C_\ell)$. So
$\PP(G)$ is enlarged by the fraction $\int D\,\Phi^\PP_B\big/\!\int\Phi^\PP_\ell\Phi^\PP_B$\,, and we
argue as in (a) with $\Phi^\PP_B$\,, which is also nonincreasing in $|u|$. Now the denominator lies
between $\int\Phi^\PP_\ell\int\Phi^\PP_B=r(\ell)V_\ell\vol\PP(B)$, by Chebyshev's inequality, and
$\Phi^\PP_\ell(0)\vol\PP(B)\le2V_\ell\vol\PP(B)$.

Finally, as in the proof of Theorem~\ref{thm:longcycle}, $\Phi^\bullet_B$ has a concave
$(d_B-1)$st root, so $\Phi^\bullet_B(u)\ge\Phi^\bullet_B(0)(1-2|u|)^{d_B-1}$, and integrating gives
$h^\bullet_B\le d_B$\,. The formula for $1-r(\ell)$ is that of Section~\ref{sec:comp}, and the bound
follows from $V_\ell-\vol\PP(C_\ell)=2^{\ell-2}/(2\ell)!$~\cite[Theorems~13 and~19]{LeeSkipper20}
and $V_\ell\ge(2/\pi^2)^\ell$ (Lemma~\ref{lem:cycmarg}).
\end{proof}

For the triangle, Corollary~\ref{cor:pendanttri} gives the sharper lower bound $1/3$ in (a). For
long cycles, the upper bounds are far below the cap of Theorem~\ref{thm:cap}: $1-r(\ell)$ decays
like $\pi^{2\ell}/(2\ell)!$, whereas $2^{\ell-1}/\ell!$ decays only like $2^\ell/\ell!$\,. For
example, let $G$ be the \emph{bouquet} of the cycles $C_3,C_4,\dots,C_L$\,, glued at one vertex, and
let $d$ be its dimension. Applying the theorem to each cycle, with $B$ the union of the others,
shows that the inequalities of $C_\ell$ cut off, and are worth, between $\tfrac12(1-r(\ell))$ and
$d\,(1-r(\ell))/r(\ell)$ of the volume. So, in one and the same graph, the contribution of a cycle
falls by a factor of order $(2\ell)^2/\pi^2$ with each additional edge.

\subsection{A long cycle as a neighbor}\label{sec:long}
The $\QQ$-marginal of a cycle has an exact spectral form.

\begin{lemma}\label{lem:cycmarg}
For $\ell\ge3$ and every vertex $v$ of $C_\ell$\,,
\[
\Phi^\QQ_{C_\ell,v}(u)=2\Bigl(\frac2{\pi^2}\Bigr)^{\!\ell}\sum_{j\ge0}
\frac{\cos^2\bigl((2j+1)\pi u\bigr)}{(2j+1)^{2\ell}}\,.
\]
Consequently,
\[
\vol\QQ(C_\ell)=\Bigl(\frac2{\pi^2}\Bigr)^{\!\ell}(1-2^{-2\ell})\,\zeta(2\ell),\qquad
\Phi^\QQ_{C_\ell,v}(0)=2\vol\QQ(C_\ell),
\]
and, for every $u$,
\[
\Bigl|\frac{\Phi^\QQ_{C_\ell,v}(u)}{\vol\QQ(C_\ell)}-2\cos^2(\pi u)\Bigr|\le4\cdot9^{-\ell}.
\]
\end{lemma}
\begin{proof}
For an edge $ij$, the $y_{ij}$ allowed by~\eqref{eq:Q} form an interval of length
$\min(x_i,x_j)-\max(0,x_i+x_j-1)=\chi(u_i,u_j)$, where
$\chi(u,t)=\tfrac12\bigl(1-|u-t|-|u+t|\bigr)$. Integrating out all $y_e$ and then the $u_w$ with
$w\ne v$ gives
\[
\Phi^\QQ_{C_\ell,v}(u)=\chi^{(\ell)}(u,u),\qquad
\chi^{(\ell)}(u,t)=\idotsint\chi(u,t_1)\,\chi(t_1,t_2)\cdots\chi(t_{\ell-1},t)\,dt_1\cdots dt_{\ell-1}\,,
\]
where the inner variables are the coordinates of the other $\ell-1$ vertices of the cycle, and
$\chi^{(\ell)}$ is the kernel of the $\ell$th power of the integral operator $\mathcal K$ on
$L^2[-\tfrac12,\tfrac12]$ with kernel $\chi$. Because $\chi$ is even in each argument, $\mathcal K$
annihilates odd functions. For even $f$, the function $\mathcal Kf$ vanishes at $u=\pm\tfrac12$,
because $\chi(\pm\tfrac12,t)=0$, and $(\mathcal Kf)''=-2f$: differentiating under the integral
sign, $(\mathcal Kf)'(u)=-\tfrac12\int\bigl(\operatorname{sgn}(u-t)+\operatorname{sgn}(u+t)\bigr)f(t)\,dt$,
and, because $\tfrac{d}{du}\int\operatorname{sgn}(u\mp t)f(t)\,dt=2f(\pm u)$, a second
differentiation gives $-\bigl(f(u)+f(-u)\bigr)=-2f(u)$. So the
eigenfunctions with nonzero eigenvalues are $\psi_j(u)=\sqrt2\cos((2j+1)\pi u)$, with eigenvalues
$\lambda_j=2/((2j+1)\pi)^2$, and they form an orthonormal basis of the even functions. Thus $\mathcal K$ is a positive-semidefinite integral operator with a continuous symmetric
kernel, and Mercer's theorem (see, for example,~\cite{RieszNagy90}) expands such a kernel, and
hence the kernel of every power, in the eigenfunctions, with uniform convergence:
$\chi^{(\ell)}(u,t)=\sum_j\lambda_j^\ell\,\psi_j(u)\,\psi_j(t)$, which is the formula. Integrating over
$u$ gives $\sum_j\lambda_j^\ell=(2/\pi^2)^\ell\sum_{j\ge0}(2j+1)^{-2\ell}$, and evaluating at $u=0$
gives twice that. For the last claim, put $\beta_j=(2j+1)^{-2\ell}$ and
$\gamma_j(u)=2\cos^2((2j+1)\pi u)\in[0,2]$, so that the normalized marginal is
$\sum_j\beta_j\gamma_j/\sum_j\beta_j$\,. Its distance from $\gamma_0$ is at most
$2\sum_{j\ge1}\beta_j/\sum_{j\ge0}\beta_j\le2\sum_{j\ge1}\beta_j$\,. Comparing the terms with
$j\ge2$ with an integral,
\[
\sum_{j\ge1}\beta_j\le9^{-\ell}+\int_1^\infty\frac{dx}{(2x+1)^{2\ell}}
=9^{-\ell}\Bigl(1+\frac{3}{2(2\ell-1)}\Bigr)\le2\cdot9^{-\ell}.
\]
\end{proof}

With $B_{2\ell}$ the Bernoulli number, Euler's formula gives
$\zeta(2\ell)=2^{2\ell-1}\pi^{2\ell}|B_{2\ell}|/(2\ell)!$, and the tangent numbers satisfy
$A_{2\ell-1}=2^{2\ell}(2^{2\ell}-1)|B_{2\ell}|/(2\ell)$ (see~\cite{Stanley10}); inserting these
identities shows that the volume formula is the closed form $\ell A_{2\ell-1}/(2^\ell(2\ell)!)$
of~\cite[Theorem~13]{LeeSkipper20}. The limit density
$2\cos^2(\pi u)$ is $2$ at $x_v=\tfrac12$ and $0$ at $x_v\in\{0,1\}$.

For two cycles of lengths $\ell$ and $\ell'$ sharing one vertex, define the \emph{gluing factor}
$\kappa(\ell,\ell')$ by $\mathrm{ratio}=\kappa(\ell,\ell')\,r(\ell)\,r(\ell')$.

\begin{theorem}\label{thm:longcycle}
Let $B$ be a series-parallel graph with a vertex $v$, let $d_B=|V(B)|+|E(B)|$, and let $G_\ell$ be
obtained by gluing $C_\ell$ to $B$ at $v$. Put
\[
L_B=\frac{\int\Phi^\PP_{B,v}(u)\cos^2(\pi u)\,du}{\int\Phi^\QQ_{B,v}(u)\cos^2(\pi u)\,du}\,.
\]
Then, for every $\ell\ge3$,
\[
\bigl|\mathrm{ratio}(G_\ell)-L_B\bigr|\le16\cdot9^{-\ell}+2\,d_B\,\bigl(1-r(\ell)\bigr),
\qquad\text{where}\qquad
1-r(\ell)\le\frac{\pi^{2\ell}}{4\,(2\ell)!}\,.
\]
In particular, $\mathrm{ratio}(G_\ell)\to L_B$\,, and the error is $O(9^{-\ell})$ with a constant
that depends on $B$ only through $d_B$\,, because $(2\ell)!$ grows faster than every exponential.
If $\rho_{B,v}$ is nondecreasing in $|u|$ and not constant, then $L_B<\mathrm{ratio}(B)$. For
$B=C_3$\,, with $\kappa_\infty=\pi^2(\pi^4+45)/(\pi^6+480)=0.975115\ldots$\,,
\[
\bigl|\kappa(3,\ell)-\kappa_\infty\bigr|\le36\cdot9^{-\ell}+30\,\bigl(1-r(\ell)\bigr)
\qquad\text{for every }\ell\ge3.
\]
\end{theorem}
\begin{proof}
Write $\Phi^\bullet_B=\Phi^\bullet_{B,v}$ and $\Phi^\bullet_\ell=\Phi^\bullet_{C_\ell,v}$ for
$\bullet\in\{\PP,\QQ\}$, let $V_\ell=\vol\QQ(C_\ell)$, let $\gamma_0(u)=2\cos^2(\pi u)$ as in the proof of
Lemma~\ref{lem:cycmarg}, and let $\alpha^\bullet=\int\Phi^\bullet_B\gamma_0$\,, so that $L_B=\alpha^\PP/\alpha^\QQ$\,; because
$\Phi^\PP_B\le\Phi^\QQ_B$\,, we have $L_B\le1$. All integrals are over $[-\tfrac12,\tfrac12]$, and
$\int\Phi^\PP_B=\vol\PP(B)$ and $\int\Phi^\QQ_B=\vol\QQ(B)$.

By~\eqref{eq:fiberprod}, $\mathrm{ratio}(G_\ell)=N/D$ with
\[
N=\frac1{V_\ell}\int\Phi^\PP_B\Phi^\PP_\ell\,,\qquad D=\frac1{V_\ell}\int\Phi^\QQ_B\Phi^\QQ_\ell\,.
\]
\emph{The denominator.} By Lemma~\ref{lem:cycmarg}, $D=\alpha^\QQ+\delta^\QQ$ with
$|\delta^\QQ|\le4\cdot9^{-\ell}\vol\QQ(B)$. Both $\Phi^\QQ_B$ and $\gamma_0$ are nonincreasing in $|u|$
(Lemma~\ref{lem:unimodal}), so Chebyshev's inequality, in the form for two functions monotone in
the same direction, gives $\alpha^\QQ\ge\int\Phi^\QQ_B\int\gamma_0=\vol\QQ(B)$. Hence
$D\ge(1-4\cdot9^{-\ell})\vol\QQ(B)\ge\tfrac12\vol\QQ(B)$.

\emph{The numerator.} We write $N=\alpha^\PP+\delta^\PP-\eta$, where
\[
\delta^\PP=\int\Phi^\PP_B\Bigl(\frac{\Phi^\QQ_\ell}{V_\ell}-\gamma_0\Bigr),\qquad
\eta=\frac1{V_\ell}\int\Phi^\PP_B\bigl(\Phi^\QQ_\ell-\Phi^\PP_\ell\bigr).
\]
By Lemma~\ref{lem:cycmarg}, $|\delta^\PP|\le4\cdot9^{-\ell}\vol\PP(B)\le4\cdot9^{-\ell}\vol\QQ(B)$.
Because $0\le\Phi^\PP_\ell\le\Phi^\QQ_\ell$ and $\Phi^\PP_B\le\Phi^\PP_B(0)$ by
Lemma~\ref{lem:unimodal}, the integrand of $\eta$ is at most $\Phi^\PP_B(0)\,(\Phi^\QQ_\ell-\Phi^\PP_\ell)$;
integrating this bound, with $\int(\Phi^\QQ_\ell-\Phi^\PP_\ell)=V_\ell-\vol\PP(C_\ell)=(1-r(\ell))V_\ell$\,,
gives $0\le\eta\le\Phi^\PP_B(0)\,(1-r(\ell))$.

\emph{The bound.} Now $\mathrm{ratio}(G_\ell)-L_B=(N-L_BD)/D=(\delta^\PP-\eta-L_B\delta^\QQ)/D$, so
\[
\bigl|\mathrm{ratio}(G_\ell)-L_B\bigr|
\le\frac{8\cdot9^{-\ell}\vol\QQ(B)+\Phi^\PP_B(0)\,(1-r(\ell))}{\tfrac12\vol\QQ(B)}\,.
\]
It remains to show that $\Phi^\PP_B(0)\le d_B\vol\QQ(B)$. We have
$\Phi^\PP_B(0)\le\Phi^\QQ_B(0)$, and, as in the proof of Lemma~\ref{lem:unimodal},
$f=\Phi^\QQ_B$ has a concave $(d_B-1)$st root on $[-\tfrac12,\tfrac12]$. Hence
$f(u)\ge f(0)(1-2|u|)^{d_B-1}$, and integrating gives $\vol\QQ(B)\ge f(0)/d_B$\,. For the bound
on $1-r(\ell)$, we use $V_\ell-\vol\PP(C_\ell)=2^{\ell-2}/(2\ell)!$ from~\cite[Theorems~13
and~19]{LeeSkipper20} and $V_\ell\ge(2/\pi^2)^\ell$ from Lemma~\ref{lem:cycmarg}.

\emph{The comparison with $\mathrm{ratio}(B)$.} The limit $L_B$ is the average of $\rho_{B,v}$
against the density proportional to $\Phi^\QQ_B\cos^2(\pi u)$, and $\mathrm{ratio}(B)$ is its
average against the density proportional to $\Phi^\QQ_B$\,. The factor $\cos^2(\pi u)$ is
strictly decreasing in $|u|$, so Chebyshev's inequality gives $L_B<\mathrm{ratio}(B)$.

\emph{The triangle.} For $B=C_3$\,, inserting the marginals of Proposition~\ref{prop:rhotri} and
integrating gives $L_{C_3}=r(3)\,\kappa_\infty$\,. Because
$\kappa(3,\ell)=\mathrm{ratio}(G_\ell)/(r(3)\,r(\ell))$,
\[
\bigl|\kappa(3,\ell)-\kappa_\infty\bigr|
\le\frac1{r(3)}\Bigl(\frac{|\mathrm{ratio}(G_\ell)-L_{C_3}|}{r(\ell)}
+L_{C_3}\,\frac{1-r(\ell)}{r(\ell)}\Bigr)
\le\frac94\Bigl(\bigl|\mathrm{ratio}(G_\ell)-L_{C_3}\bigr|+1-r(\ell)\Bigr),
\]
because $r(3)=2/3$, $r(\ell)\ge2/3$ for $\ell\ge3$ (by the bound on $1-r(\ell)$ for $\ell\ge4$),
and $L_{C_3}\le1$. With $d_{C_3}=6$, this is at most
$36\cdot9^{-\ell}+\tfrac{117}4(1-r(\ell))$.
\end{proof}

A long cycle loses almost nothing to its own odd-cycle inequalities, because $r(\ell)\to1$, yet it
is not a harmless neighbor: it acts on $B$ as the fixed vertex weight $\cos^2(\pi u_v)$, which is
concentrated where $\rho_{B,v}$ is smallest.

\subsection{Large flowers}\label{sec:flowers}
When many copies of a graph $B$ meet at one vertex $v$, the fibers concentrate at $x_v=\tfrac12$\,,
and the value $\rho_{B,v}(0)$ governs the ratio. The proof is an instance of the Laplace method:
for large $k$, the integral of the $k$th power of a function is governed by a neighborhood of the
maximum of the function.

\begin{theorem}\label{thm:bigflower}
Let $B$ be a series-parallel graph with a cycle and a vertex $v$, let $d_B=|V(B)|+|E(B)|$, and
let $F_k$ consist of $k$ copies of $B$ glued at $v$ and otherwise disjoint. Then, for every
$k\ge1$,
\[
\frac{\rho_{B,v}(0)^k}{4k(d_B-1)}\le\mathrm{ratio}(F_k)\le4k(d_B-1)\,\rho_{B,v}(0)^k.
\]
In particular, $\mathrm{ratio}(F_k)^{1/k}\to\rho_{B,v}(0)<1$, with relative error
$O(\log(k\,d_B)/k)$.
\end{theorem}
\begin{proof}
By~\eqref{eq:fiberprod},
$\mathrm{ratio}(F_k)=\int(\Phi^\PP_{B,v})^k\big/\!\int(\Phi^\QQ_{B,v})^k$\,. By
Lemma~\ref{lem:unimodal}, $\Phi^\PP_{B,v}$ and $\Phi^\QQ_{B,v}$ attain their maxima $M^\PP$ and
$M^\QQ$ at $u=0$, and $M^\PP/M^\QQ=\rho_{B,v}(0)<1$. As in the proof of Theorem~\ref{thm:longcycle},
both functions have a concave $(d_B-1)$st root, so on $[-\delta,\delta]$ they are at least
$(1-2\delta)^{d_B-1}$ times their maxima. Integrating the $k$th powers over $[-\delta,\delta]$ for
the lower bounds and over $[-\tfrac12,\tfrac12]$ for the upper bounds gives
\[
2\delta\,(1-2\delta)^{k(d_B-1)}\rho_{B,v}(0)^k\le\mathrm{ratio}(F_k)
\le\frac{\rho_{B,v}(0)^k}{2\delta\,(1-2\delta)^{k(d_B-1)}}\,.
\]
We choose $2\delta=1/(k(d_B-1))$ and use $(1-1/n)^n\ge1/4$ for $n\ge2$; here
$n=k(d_B-1)\ge5$, because a graph with a cycle has $d_B\ge6$.
\end{proof}

For cycles, $\rho_{C_\ell,v}(0)$ has a closed form.

\begin{proposition}\label{prop:rhocenter}
For $\ell\ge3$ and every vertex $v$ of $C_\ell$\,,
\[
\Phi^\QQ_{C_\ell,v}(0)-\Phi^\PP_{C_\ell,v}(0)=\frac{1}{2^\ell\,(2\ell-1)!}\binom{2\ell-2}{\ell-1}
\qquad\text{and}\qquad
\rho_{C_\ell,v}(0)=1-\frac{1}{A_{2\ell-1}}\binom{2\ell-2}{\ell-1}.
\]
\end{proposition}
\begin{proof}
By Theorem~\ref{thm:padberg}, $\QQ(C_\ell)\setminus\PP(C_\ell)$ is the union of the
$2^{\ell-1}$ sets on which $\OC(A)$ fails, one for each odd $A\subseteq E(C_\ell)$, and they are
disjoint (proof of Theorem~\ref{thm:cap}). Switching at vertices other than $v$, and at $v$, which fixes
$u_v=0$, permutes these sets (Lemma~\ref{lem:switch}), and the switchings act transitively on
the odd sets $A$. So each contributes the same volume to the fiber over $u_v=0$.

Fix $A$, and use the notation of the proof of Theorem~\ref{thm:cap}, with $v_0=v$; in the fiber,
$\tilde u_0=\tilde u_\ell=0$. As shown there, the set on which $\OC(A)$ fails is described by
$\sum_jw_j<1$ and the lower bounds $w_j\ge|\tilde u_j-\tilde u_{j-1}|$. With $h_j=\tilde u_j-\tilde u_{j-1}$\,, so that $\sum_jh_j=0$, the fiber volume of the set in the
coordinates $(\tilde u,w)$ is therefore
\[
\int_{\sum_jh_j=0}\ \frac{(1-\|h\|_1)_+^\ell}{\ell!}\,d(h_1,\dots,h_{\ell-1}).
\]
The region $\{\sum_jh_j=0,\ \|h\|_1\le t\}$ has volume $b_\ell t^{\ell-1}$. To compute $b_\ell$\,,
note that measuring a region of the hyperplane $\{\sum_jh_j=0\}$ in the coordinates
$(h_1,\dots,h_{\ell-1})$ gives the same volume as measuring it in any other $\ell-1$ of the
coordinates, because the change between the two parametrizations is linear with determinant
$\pm1$. We split the region by the set of positive coordinates, of size $p$, and parametrize by
the sum $s$ of the positive coordinates together with all but one positive and all but one
negative coordinate; the positive and the negative coordinates then range over simplices, and
\[
b_\ell=\sum_{p=1}^{\ell-1}\binom\ell p\int_0^{1/2}\frac{s^{p-1}}{(p-1)!}\,
\frac{s^{\ell-p-1}}{(\ell-p-1)!}\,ds
=\frac{1}{2^{\ell-1}(\ell-1)!}\binom{2\ell-2}{\ell-1},
\]
by Vandermonde's identity. Hence the integral is
\[
\frac{b_\ell}{\ell!}\int_0^1(1-t)^\ell\,d\bigl(t^{\ell-1}\bigr)
=\frac{b_\ell\,(\ell-1)!}{(2\ell-1)!}
=\frac{1}{2^{\ell-1}(2\ell-1)!}\binom{2\ell-2}{\ell-1}.
\]
Multiplying by the $2^{\ell-1}$ sets $A$ and by $2^{-\ell}$, which converts the $z_e$ into the
$y_e$\,, gives the first formula. By Lemma~\ref{lem:cycmarg},
$\Phi^\QQ_{C_\ell,v}(0)=2\vol\QQ(C_\ell)=A_{2\ell-1}/(2^\ell(2\ell-1)!)$, which gives the second.
\end{proof}

In a large flower, each cycle lowers the ratio strictly more than on its own.

\begin{corollary}\label{cor:bigflower}
For every $\ell\ge3$, $\rho_{C_\ell,v}(0)<r(\ell)$. Hence the flower $F_k$ of $k$ copies of
$C_\ell$ satisfies $\mathrm{ratio}(F_k)<r(\ell)^k$ for all large $k$, and
\[
\frac{1-\rho_{C_\ell,v}(0)}{1-r(\ell)}=\frac{\ell}{4^{\ell-1}}\binom{2\ell-2}{\ell-1}
\sim\sqrt{\ell/\pi}\,.
\]
\end{corollary}
\begin{proof}
The quotient follows from Proposition~\ref{prop:rhocenter} and the formula for $r(\ell)$. With
$c_m=\binom{2m}{m}(m+1)/4^m$, we have $c_1=1$ and $c_{m+1}/c_m=(2m+1)(m+2)/(2(m+1)^2)>1$ for
$m\ge1$, so the quotient, which equals $c_{\ell-1}$\,, exceeds $1$ for $\ell\ge3$. For the asymptotic
form, Stirling's formula gives $\binom{2m}m=\bigl(4^m/\sqrt{\pi m}\bigr)\bigl(1+O(1/m)\bigr)$, so
$c_m=\bigl((m+1)/\sqrt{\pi m}\bigr)\bigl(1+O(1/m)\bigr)\sim\sqrt{m/\pi}$. The claim about $F_k$
follows from Theorem~\ref{thm:bigflower}.
\end{proof}

For the triangle, the ratio of a flower of $k$ triangles decays like $(5/8)^k$, against
$(2/3)^k$ for $k$ disjoint triangles.

\subsection{What short odd cycles buy}\label{sec:short}
In practice, one separates odd-cycle inequalities only for short cycles, often only for
triangles; see~\cite[Section~4]{LSS18}. For $L\ge3$, let $\PP_{\le L}(G)$ be $\QQ(G)$ intersected with
the odd-cycle inequalities of the cycles of $G$ of length at most $L$; thus $\PP_{\le3}(G)$ adds
only the triangle inequalities. We ask how much of the gap between $\QQ(G)$ and $\PP(G)$ these
inequalities can close.

\begin{proposition}\label{prop:chordless}
Let $G$ be series-parallel and $L\ge3$. Then $\vol\PP_{\le L}(G)=\vol\PP(G)$ if and only if every
chordless cycle of $G$ has length at most $L$.
\end{proposition}
\begin{proof}
Let $C$ be a cycle with a chord $f$, which splits $C$ into two shorter cycles $C_1$ and $C_2$
through $f$, and let $A\subseteq E(C)$ be odd. Exactly one of $|A\cap E(C_1)|$ and
$|A\cap E(C_2)|$ is even; add $f$ to that part, which makes both parts odd. By
Lemma~\ref{lem:oc}, the sum of the two resulting odd-cycle inequalities is the inequality
$\OC(A)$ for $C$ plus the identity $(1-z_f)+z_f=1$. By induction on the length, every odd-cycle
inequality is implied by those of chordless cycles, and the ``if'' direction follows from
Theorem~\ref{thm:padberg}.

Conversely, let $C$ be a chordless cycle of length greater than $L$, and $f\in E(C)$. Every
other cycle $C'$ of $G$ has at least two edges outside $C$: if it had only one, that edge would
join two vertices of $C$ and be a chord. Let $p$ be the point with all $u_w=0$, with $z_f=1-\varepsilon$,
with $z_e=\varepsilon$ for $e\in E(C)\setminus f$, and with $z_e=\tfrac12$ for all other edges, where
$0<\varepsilon<1/|C|$. By Lemma~\ref{lem:edge}, $p$ lies in the interior of $\QQ(G)$, and by
Lemma~\ref{lem:oc} it violates $\OC(\{f\})$ for $C$. For a cycle $C'\ne C$ and an odd
$A'\subseteq E(C')$, the left-hand side of Lemma~\ref{lem:oc} receives $\tfrac12$ from each edge
outside $C$ and a positive amount from each edge of $C$, so it exceeds $1$ unless $C'$ has exactly
two edges outside $C$ and none in $C$, which is impossible. So $p$ satisfies every odd-cycle
inequality of every cycle of length at most $L$ strictly, and a neighborhood of $p$ lies in
$\PP_{\le L}(G)\setminus\PP(G)$.
\end{proof}

At the other extreme, the inequalities of short cycles can leave almost all of the gap open,
even for cacti.

\begin{theorem}\label{thm:worst}
For every $L\ge3$,
\[
\inf_G\frac{\vol\PP(G)}{\vol\PP_{\le L}(G)}=0,
\]
where $G$ ranges over connected series-parallel graphs, and the infimum is approached
exponentially fast in $d$. More precisely, for every $\ell>L$, the flower $F_k$ of $k$ copies of
$C_\ell$ has $d=1+k(2\ell-1)$, satisfies $\PP_{\le L}(F_k)=\QQ(F_k)$, and
\[
\frac{\rho_\ell^k}{4k(2\ell-1)}\le\frac{\vol\PP(F_k)}{\vol\PP_{\le L}(F_k)}\le4k(2\ell-1)\,\rho_\ell^k\,,
\]
where $\rho_\ell=1-\binom{2\ell-2}{\ell-1}/A_{2\ell-1}$\,.
For the triangle inequalities and $\ell=4$, the ratio is at most $28k\,(63/68)^k$, that is,
$(0.98915\ldots)^{d+o(d)}$.
\end{theorem}
\begin{proof}
Every cycle of $F_k$ is one of the copies of $C_\ell$\,, of length $\ell>L$, so $\PP_{\le L}(F_k)$
has no odd-cycle inequalities beyond $\QQ(F_k)$. The bounds follow from
Theorem~\ref{thm:bigflower}, with $d_{C_\ell}=2\ell$, and Proposition~\ref{prop:rhocenter}.
\end{proof}

So, in the worst case, the triangle inequalities close no fixed fraction of the gap between
$\QQ(G)$ and $\PP(G)$, and neither do the odd-cycle inequalities of cycles up to any fixed
length. By Proposition~\ref{prop:chordless}, they close all of it exactly when every chordless
cycle is short. In between, what matters is the chordless cycles that are too long to be
separated. Omitting any one of them loses little: by Theorem~\ref{thm:cap}, a cycle of length $\ell$ never cuts
off more than $2^{\ell-1}/\ell!$ of $\QQ(G)$, and in a large flower it removes the fraction
$\binom{2\ell-2}{\ell-1}/A_{2\ell-1}$\,, although, by Corollary~\ref{cor:bigflower}, that is about
$\sqrt{\ell/\pi}$ times what the same cycle removes on its own. Short cycles are therefore worth
far more per cycle, but the small fractions that many long cycles cut off compound, and no bound in terms of $L$
alone survives.

% ---------------------------------------------------------------------------
\section{Concluding remarks}\label{sec:concl}

From Theorem~\ref{thm:main}, we answer the question of~\cite[Section~8]{LeeSkipper20}, which
Lee, Skipper, and Speakman restated as~\cite[Problem~20, ``Lobster Roll'']{LSS18}. Series-parallel graphs are the
natural limit of this method: if $G$ has a $K_4$ minor, then $\PP(G)$ has facets beyond the
odd-cycle inequalities, and the parity distances no longer determine membership. The coordinates
behind the algorithm also bound the volume that odd-cycle inequalities can cut off: in every
graph, a cycle of length $\ell$ cuts off at most $2^{\ell-1}/\ell!$ of $\QQ(G)$
(Theorem~\ref{thm:cap}), yet many long cycles together can leave $\PP(G)$ an exponentially small
part of the relaxation defined by short cycles (Theorem~\ref{thm:worst}). We would like to highlight
two open questions:
\begin{itemize}[nosep]
\item deciding whether the approach extends to graphs of bounded treewidth, for relaxations of
$\PP(G)$ defined by odd-cycle inequalities alone;
\item deciding whether $\vol\mathrm{Cut}(G)$ can be computed in polynomial time for every graph $G$
without a $K_5$ minor, for which $\mathrm{Cut}(G)$ is still described by the inequalities of
Section~\ref{sec:cutG}~\cite{BarahonaMahjoub86}, a question raised in~\cite[Section~7]{ADD26};
Theorem~\ref{thm:cutG} settles the series-parallel case.
\end{itemize}

\section*{Acknowledgment}
This work was supported in part by the Office of Naval Research under grant N00014-24-1-2694. The
author thanks David Avis and Antoine Deza for rekindling his interest in this topic.

\medskip
\noindent {\bf Declaration of generative AI and AI-assisted technologies in the manuscript preparation process.}
During the preparation of this work, the author used Claude to assist in developing and checking
the work. The author takes full responsibility for the content of the article.

\bibliographystyle{alpha}
\bibliography{refs}

@article{BarahonaMahjoub86,
  author  = {Francisco Barahona and Ali Ridha Mahjoub},
  title   = {On the cut polytope},
  journal = {Mathematical Programming},
  volume  = {36},
  number  = {2},
  pages   = {157--173},
  year    = {1986},
  note    = {\url{https://doi.org/10.1007/BF02592023}}
}

@incollection{deBoor76,
  author    = {Carl de Boor},
  title     = {Splines as linear combinations of {B}-splines: a survey},
  booktitle = {Approximation Theory {II}},
  editor    = {G. G. Lorentz and C. K. Chui and L. L. Schumaker},
  publisher = {Academic Press},
  address   = {New York},
  pages     = {1--47},
  year      = {1976}
}

@article{DeSimone90,
  author  = {De Simone, Caterina},
  title   = {The cut polytope and the {B}oolean quadric polytope},
  journal = {Discrete Mathematics},
  volume  = {79},
  number  = {1},
  pages   = {71--75},
  year    = {1990},
  note    = {\url{https://doi.org/10.1016/0012-365X(90)90056-N}}
}

@book{DezaLaurent97,
  author    = {Michel Marie Deza and Monique Laurent},
  title     = {Geometry of Cuts and Metrics},
  series    = {Algorithms and Combinatorics},
  volume    = {15},
  publisher = {Springer},
  address   = {Berlin},
  year      = {1997},
  note      = {\url{https://doi.org/10.1007/978-3-642-04295-9}}
}

@article{DittmerPak20,
  author  = {Samuel Dittmer and Igor Pak},
  title   = {Counting linear extensions of restricted posets},
  journal = {Electronic Journal of Combinatorics},
  volume  = {27},
  number  = {4},
  pages   = {Paper 4.48},
  year    = {2020},
  note    = {\url{https://doi.org/10.37236/8552}}
}

@article{Hanson72,
  author  = {Denis Hanson},
  title   = {On the product of the primes},
  journal = {Canadian Mathematical Bulletin},
  volume  = {15},
  number  = {1},
  pages   = {33--37},
  year    = {1972},
  note    = {\url{https://doi.org/10.4153/CMB-1972-007-7}}
}

@inproceedings{KHNK16,
  author    = {Kustaa Kangas and Teemu Hankala and Teppo Niinim{\"a}ki and Mikko Koivisto},
  title     = {Counting linear extensions of sparse posets},
  booktitle = {Proceedings of the Twenty-Fifth International Joint Conference on Artificial Intelligence ({IJCAI}~2016)},
  pages     = {603--609},
  year      = {2016},
  note      = {\url{https://www.ijcai.org/Abstract/16/092}}
}

@article{KLS97,
  author  = {Chun-Wa Ko and Jon Lee and Einar Steingr{\'{\i}}msson},
  title   = {The volume of relaxed {B}oolean-quadric and cut polytopes},
  journal = {Discrete Mathematics},
  volume  = {163},
  number  = {1--3},
  pages   = {293--298},
  year    = {1997},
  note    = {\url{https://doi.org/10.1016/0012-365X(95)00343-U}}
}

@article{LeeMorris94,
  author  = {Jon Lee and Morris, Jr., Walter D.},
  title   = {Geometric comparison of combinatorial polytopes},
  journal = {Discrete Applied Mathematics},
  volume  = {55},
  number  = {2},
  pages   = {163--182},
  year    = {1994},
  note    = {\url{https://doi.org/10.1016/0166-218X(94)90006-X}}
}

@article{LeeSkipper20,
  author  = {Jon Lee and Daphne Skipper},
  title   = {Volume computation for sparse {B}oolean quadric relaxations},
  journal = {Discrete Applied Mathematics},
  volume  = {275},
  pages   = {79--94},
  year    = {2020},
  note    = {\url{https://doi.org/10.1016/j.dam.2018.10.038}}
}

@article{LSS18,
  author  = {Jon Lee and Daphne Skipper and Emily Speakman},
  title   = {Algorithmic and modeling insights via volumetric comparison of polyhedral relaxations},
  journal = {Mathematical Programming},
  volume  = {170},
  number  = {1},
  pages   = {121--140},
  year    = {2018},
  note    = {\url{https://doi.org/10.1007/s10107-018-1272-6}}
}

@article{Micchelli80,
  author  = {Charles A. Micchelli},
  title   = {A constructive approach to {K}ergin interpolation in {$R^k$}: multivariate {B}-splines and {L}agrange interpolation},
  journal = {Rocky Mountain Journal of Mathematics},
  volume  = {10},
  number  = {3},
  pages   = {485--497},
  year    = {1980},
  note    = {\url{https://doi.org/10.1216/RMJ-1980-10-3-485}}
}

@article{Padberg89,
  author  = {Manfred Padberg},
  title   = {The {B}oolean quadric polytope: some characteristics, facets and relatives},
  journal = {Mathematical Programming},
  volume  = {45},
  number  = {1},
  pages   = {139--172},
  year    = {1989},
  note    = {\url{https://doi.org/10.1007/BF01589101}}
}

@article{Stanley86,
  author  = {Richard P. Stanley},
  title   = {Two poset polytopes},
  journal = {Discrete \& Computational Geometry},
  volume  = {1},
  number  = {1},
  pages   = {9--23},
  year    = {1986},
  note    = {\url{https://doi.org/10.1007/BF02187680}}
}

@article{VTL82,
  author  = {Jacobo Valdes and Robert E. Tarjan and Eugene L. Lawler},
  title   = {The recognition of series parallel digraphs},
  journal = {{SIAM} Journal on Computing},
  volume  = {11},
  number  = {2},
  pages   = {298--313},
  year    = {1982},
  note    = {\url{https://doi.org/10.1137/0211023}}
}

@inproceedings{EGKO16,
  author    = {Eduard Eiben and Robert Ganian and Kustaa Kangas and Sebastian Ordyniak},
  title     = {Counting linear extensions: parameterizations by treewidth},
  booktitle = {24th Annual European Symposium on Algorithms ({ESA} 2016)},
  series    = {{LIPIcs}},
  volume    = {57},
  pages     = {39:1--39:18},
  year      = {2016},
  note      = {\url{https://doi.org/10.4230/LIPIcs.ESA.2016.39}}
}

@incollection{Mohring89,
  author    = {Rolf H. M{\"o}hring},
  title     = {Computationally tractable classes of ordered sets},
  booktitle = {Algorithms and Order},
  editor    = {I. Rival},
  publisher = {Kluwer Academic Publishers},
  address   = {Dordrecht},
  pages     = {105--193},
  year      = {1989},
  note      = {\url{https://doi.org/10.1007/978-94-009-2639-4_4}}
}

@book{HLP52,
  author    = {Godfrey Harold Hardy and John Edensor Littlewood and George P{\'o}lya},
  title     = {Inequalities},
  edition   = {Second},
  publisher = {Cambridge University Press},
  year      = {1952}
}

@article{Gardner02,
  author  = {Richard J. Gardner},
  title   = {The {B}runn-{M}inkowski inequality},
  journal = {Bulletin of the American Mathematical Society},
  volume  = {39},
  number  = {3},
  pages   = {355--405},
  year    = {2002},
  note    = {\url{https://doi.org/10.1090/S0273-0979-02-00941-2}}
}

@incollection{Stanley10,
  author    = {Richard P. Stanley},
  title     = {A survey of alternating permutations},
  booktitle = {Combinatorics and Graphs},
  series    = {Contemporary Mathematics},
  volume    = {531},
  pages     = {165--196},
  publisher = {American Mathematical Society},
  year      = {2010},
  note      = {\url{https://doi.org/10.1090/conm/531/10466}}
}

@book{RieszNagy90,
  author    = {Frigyes Riesz and B{\'e}la Sz.-Nagy},
  title     = {Functional Analysis},
  publisher = {Dover Publications},
  year      = {1990},
  note      = {Reprint of the 1955 English translation}
}

@misc{ADD26,
  author = {David Avis and Luc Devroye and Antoine Deza},
  title  = {On the volumes of the elliptope, metric polytope, and cut polytope},
  year   = {2026},
  note   = {arXiv:2604.16735, \url{https://arxiv.org/abs/2604.16735}}
}

\end{document}